\documentclass[12pt]{amsart}

\usepackage{amsmath,amsthm,amsfonts,amssymb}

\theoremstyle{plain}
\newtheorem{thm}{Theorem}[section]
\newtheorem{prop}[thm]{Proposition}
\newtheorem{lem}[thm]{Lemma}
\newtheorem{cor}[thm]{Corollary}

\theoremstyle{definition}
\newtheorem{definition}[equation]{Definition}

\newtheorem{remark}[equation]{Remark}
\newtheorem{dfn}[thm]{Definition}

\theoremstyle{remark}

\numberwithin{equation}{section}

\renewcommand{\phi}{\varphi}

\newcommand{\GL}{{\rm GL}}

\DeclareMathSymbol{\sdp}{\mathbin}{AMSb}{"6F}

\newcommand{\diag}{\mathrm{diag}}

\newcommand\bb[1]{{\text{\bf#1}}}

\newcommand\bbz{\mathbb{Z}} 

\newcommand\bbr{\mathbb{R}} 
\newcommand\bbc{\mathbb{C}}

\newcommand\bbp{\bb{P}}
\newcommand\bbi{\bb{I}}

\renewcommand\gg{\mathfrak{g}}
\newcommand\gh{\mathfrak{h}} 
 
\newcommand\gu{\mathfrak{u}} 
\newcommand\gs{\mathfrak{s}}

\newcommand\ca{\mathcal}
\newcommand\cagm{{\ca G \ca M}_{n,\delta_0} }

\begin{document}
\baselineskip=16pt
\title[Grassmannian framed bundles and generalized parabolic structures] {Grassmannian framed bundles and generalized parabolic structures}

\author[U. Bhosle]{Usha Bhosle}

\address{School of Mathematics, Tata Institute of Fundamental Research,
Homi Bhabha Road, Bombay 400005, India}

\email{usha@math.tifr.res.in}

\author[I. Biswas]{Indranil Biswas}

\address{School of Mathematics, Tata Institute of Fundamental Research,
Homi Bhabha Road, Bombay 400005, India}

\email{indranil@math.tifr.res.in}

\author[J. Hurtubise]{Jacques Hurtubise}

\address{Department of Mathematics, McGill University, Burnside
Hall, 805 Sherbrooke St. W., Montreal, Que. H3A 2K6, Canada}

\email{jacques.hurtubise@mcgill.ca}



\date{}

\begin{abstract}
We build compact moduli spaces of Grassmannian framed bundles over a 
Riemann surface, essentially replacing a group by a 
bi-equivariant
compactification. We do this both in the algebraic and symplectic 
settings, and prove a Hitchin-Kobayashi correspondence between the two. 
The spaces are universal spaces for parabolic bundles (in the sense that all of the moduli can be obtained as quotients), and the reduction to 
parabolic bundles commutes with the correspondence. An analogous 
correspondence is outlined for the generalized parabolic bundles of Bhosle.

\end{abstract}

\maketitle

\section{Introduction}

The Hitchin-Kobayashi correspondence, which 
establishes on a compact K\"ahler (or even more general) manifold a 
bicontinuous 
correspondence between on one hand unitary bundles equipped with an irreducible 
connection (and possibly auxiliary fields) satisfying a suitable curvature 
condition, and on the other, stable holomorphic bundles or stable 
holomorphic pairs, triples, etc, is by now a well established paradigm, 
proven over the years in increasing degrees of generality, by 
Narasimhan-Seshadri \cite{NS}, Mehta-Seshadri \cite{MS}, 
Donaldson \cite{Do1, Do2}, Uhlenbeck-Yau \cite{UY}, Simpson 
\cite{Si}; good references and an overview can be found in 
L\"ubke-Teleman \cite{LuTe}. This correspondence has been invaluable, 
both 
for understanding the holomorphic moduli, and in understanding the 
moduli of connections satisfying the curvature condition (for example, 
anti-self-duality in complex dimension two).

In dimension one, the curvature condition is generally one of flatness, 
or, for non-zero degree, of constant central curvature; the particular 
correspondences that concern us first are the classical ones, which 
served as models for the others:

\begin{itemize} \item The
correspondence on a closed Riemann surface between stable holomorphic vector bundles and flat (or 
constant central curvature) unitary connections. The holomorphic stable bundles 
then get linked to unitary representations of the fundamental group of 
the surface. See Narasimhan-Seshadri \cite{NS}, Donaldson \cite{Do1}.

\item The correspondence on a Riemann surface with marked points between stable holomorphic 
vector bundles with 
parabolic structures at marked points, and flat unitary connections on 
the complement of the marked points with fixed conjugacy classes for the 
holonomy around the marked points. See Mehta-Seshadri \cite{MS}, Biquard 
\cite{Bi}, Poritz \cite{Po}.
\end{itemize}

For these cases, there are two ``universal", moduli spaces, one holomorphic, one symplectic, which in some sense 
contain all of the parabolic spaces, in  that the parabolic moduli can be 
obtained as quotients of these spaces in a uniform way. The first, on the 
algebraic or holomorphic side, is the space of framed bundles, i.e., the 
space of pairs of (bundles, trivializations at the marked points); the 
second, on the unitary or symplectic side, are the extended moduli 
spaces defined by Jeffrey \cite{Je}. By rights, these spaces should also 
correspond under the Hitchin-Kobayashi correspondence, but so far this 
has not been clear.

Both of these spaces also present the difficulty of being inherently 
non-compact. On the algebraic side, there is a compactification, but by 
sheaves \cite{HL}; on the symplectic side, as well as non-compactness, there is the problem that the 
symplectic form 
can degenerate. (One can get a compact space on the symplectic side, at the price of considering quasi-Hamiltonian structures, as in \cite{AMM}.) 

We will remedy all of these problems by replacing the framings 
with their graphs in a Grassmannian modelled on the Grassmannian of 
$n$-planes in $\bbc^{2n}$. This latter Grassmannian is a particularly nice 
smooth compactification of $\GL(n)$; in particular it is equivariant 
under both the left and right actions of $\GL(n)$, represented by 
the embeddings of $\GL(n)$ into $\GL(2n)$ as subgroups 
$\GL(n)\times\{\bbi\}, \{\bbi\}\times \GL(n)$. 

It is the purpose of this paper to construct the moduli spaces $\cagm$ of pairs
(bundles, Grassmannian framing), as well as the analogous spaces $GM_{n, \delta_0}$ on 
the symplectic side. We will show:
\begin{itemize} 
\item that they are compact; on the holomorphic side, this space 
only involves bundles, and on the symplectic side, it 
is symplectic 
where it is reasonable to expect this (i.e., over the locus where a moment map is a submersion);
\item that there is a bijective Hitchin-Kobayashi correspondence $C: GM_{n, \delta_0}\rightarrow \cagm$ relating them;
\item that on the symplectic side, there is a Hamiltonian action of $U(n)^\ell$, and on the holomorphic side, a holomorphic action of $Gl(n,\bbc)^\ell$ such that one can obtain the various moduli   spaces of parabolic bundles from our spaces 
as either  symplectic quotients under $U(n)^\ell$ or as holomorphic quotients under $Gl(n,\bbc)^\ell$; 
this process commutes with the  Hitchin-Kobayashi correspondence; see the end of section 4 below.
\item that on the symplectic side again, the moment map for the $U(n)^\ell$-action is the conjugacy class of a logarithm of the holonomy around the puncture. Our construction will give us all representations of the fundamental group of the punctured Riemann surface,  along with some framing information, with no redundancy, apart from the cases where the holonomy around the puncture is minus the identity, in which cases there are some extra spaces which are added. 
\end{itemize}

It is the relationship to parabolic bundles that explains the use of the Grassmannian compactification of $Gl(n,\bbc)$. There are, as recent work of Martens and Thaddeus \cite{MaTh} has emphasized quite beautifully, quite a few viable compactifications of $Gl(n,\bbc)$ to choose from. Apart from the convenience of using an explicit space, there is a good geometric reason for choosing this particular compactification. This is best seen from the symplectic point of view. Indeed, the image of the moment map for the action of $U(n)$ on the Grassmannian that extends the left action on $Gl(n,\bbc)$ is a $U(n)$ orbit of the diagonal matrices with entries in $[-1/2, 1/2]$; the Kirwan polytope for the moment map is the set $-1/2\leq \lambda_1\leq \lambda_2\leq\cdots \leq \lambda_n\leq1/2$. Normalizing by $2\pi \sqrt{-1}$, this corresponds naturally to the fundamental alcove of $U(n)$, and so to the set of conjugacy classes of holonomies around the puncture, the only redundancy being that $\pm 1/2$ represent the same holonomy of $-1$. In other words this Grassmannian gives through its moment map all the possible weights for a parabolic moduli space, and only once, again with a caveat about $\pm 1/2$. Results of Alexeev and Brion \cite{AB} tell us that the Kirwan polytope essentially determines the bi-equivariant compactification of $Gl(n,\bbc)$, confirming that the Grassmannian is in essence the right choice. In some sense, the moduli spaces we construct are, in essence, the ``smallest" compact symplectic/holomorphic spaces which ``contain" all unitary representations of the fundamental group of the  punctured curve, as unitary/general linear quotients.

If, instead of considering the Grassmannian 
of $n$-planes in the direct sum $E_p\oplus \bbc^n$ of the fiber $E_p$ of 
a 
vector bundle $E$ and a fixed copy of $\bbc^n$, one now chooses a pair of  points $p$ and $q$, and considers the 
Grassmannian of 
$n$-planes in $E_p\oplus E_q$,  we obtain an example of the generalized parabolic 
bundles of Bhosle \cite{Bh2}; the generic element of this space corresponds  to bundles over the nodal curve obtained by identifying $p$ and $q$, by a simple glueing.    Again, one can consider this both from  the
symplectic and holomorphic points of view, with a Hitchin-Kobayashi 
correspondence connecting them; also, one can construct from both 
points of view
the moduli of bundles on a nodal curve. We summarise this in section 5 of this paper. Indeed, it was the question of finding out what the Narasimhan-Seshadri theorem looked like for nodal curves which was the initial motivation for this work.

We note in passing that another set of universal spaces for parabolic 
moduli, this time with trivializations parametrized by a torus instead of $\GL(n,\bbc)$, was given by Hurtubise, Jeffrey and Sjamaar in \cite{HuJeS}; the quotient there is by a torus, but one must include some non-locally free sheaves in the space.

There should be equivalent correspondences for arbitrary reductive 
groups; these should be particularly interesting, as the bi-equivariant 
compactifications of these groups have also recently been constructed, using 
bundles on rational curves, by Martens and Thaddeus \cite{MaTh}. We will 
return to these elsewhere.

Section \ref{sec2} is devoted to the construction of the moduli of 
Grassmannian 
framed bundles; Section \ref{sec3} recalls Jeffrey's construction of the 
extended 
moduli spaces, and uses it to construct a Grassmannian analog. In 
Section \ref{sec4} we prove the correspondence. In Section \ref{sec5}, 
we show how the same ideas extend to Bhosle's generalized parabolic 
structures. Section \ref{sec6} gives examples.

\medskip
\noindent
\textbf{Acknowledgements.}\, The authors would like to thank the Kerala 
School of Mathematics for 
their hospitality during the discussions which launched this project.

\section{Moduli of Grassmannian-framed bundles}\label{sec2}

\subsection{ Definitions and notation}
Let $X$ be an irreducible smooth complex projective curve of genus $g$, 
with $\ell$ marked (ordered) points $p_1,\cdots ,p_\ell$.
Throughout this paper, $E$ will denote a 
vector bundle of rank $n$ over $X$. Initially, the degree
of $E$ is $\delta_0$ with $-[(\ell n-1)/2]\, \leq\, \delta_0
\, \leq\, [(\ell n-1)/2]$, where $[t]$ denotes the integral part of $t$.
(This rather odd choice of range for the degrees is due to an eventual 
link to bundles with parabolic structure; in the course of the moduli 
construction, this degree will be increased into a stable range, as the 
result of twisting by a line bundle of positive degree, as is usual for 
moduli constructions.)
The Grassmannian parametrizing linear subspaces 
of $E_{p_i}\oplus \bbc^n$ of dimension $n$ will be denoted by
${\rm Gr}_n( E_{p_i}\oplus \bbc^n)$; if a
subspace of $E_{p_i}\oplus \bbc^n$
is in general position, it is the graph of a trivialization of $E$
at $p_i$, or more generally, the graph of a linear map from
$\bbc^n$ to $E_{p_i}$. Throughout, $g_i$ will denote
an element of ${\rm Gr}_n( E_{p_i}\oplus \bbc^n)$.
Set $\vec{g}\,=\, (g_1,\cdots ,g_\ell)$.

There are two numbers associated to an element $g_i$ of the 
Grassmannian ${\rm Gr}_n( E_{p_i}\oplus \bbc^n)$:
\begin{align} s_i&= \dim(g_i\cap E_{p_i})\\ t_i&= \dim(g_i\cap 
\bbc^n) = \dim(E_{p_i}/\Pi(g_i))\, ;\end{align}
here $\Pi(g_i)$ is the projection of $g_i$ to $E_{p_i}$. We note that 
$g_i$ is the graph of a trivialization if $s_i= t_i = 0$; 
there is the obvious bound $s_i+t_i\leq n$. 

 We call the pair $(E,\vec{g})$ a {\it Grassmannian framed vector 
bundle}. For a subbundle 
$E'\subset E$, let $g'_i:= g_i \bigcap (E'_{p_i}\oplus \bbc^n)$. 
Define
$$
s'_i\,:=\, \dim(g_i\cap E'_{p_i}) \, ~\,\text{~and~}\, ~\, t'_i\, 
:=\, \dim(E'_{p_i}/(\Pi(g_i)\cap E'_{p_i}))\, .
$$

\begin{dfn}
 We call $(E,\vec{g})$ {\it semistable} if the following conditions 
hold: 
\begin{enumerate}
\item the inequality
\begin{equation}\label{1}
 \sum_{i=1}^{\ell}s_i\, \leq \, \frac{\ell n}{2} - \delta_0;
\end{equation}

\item the inequality
\begin{equation}\label{2}
\sum_{i=1}^{\ell}t_i\,\leq \, \frac{\ell n}{2} + \delta_0;
\end{equation}

\item for every subbundle $E'$ of $E$,
\begin{equation}\label{3}
0\, \leq\, \frac{\delta_0}{n} - 
 \frac{\delta'_0}{n'} + \sum_i\Bigl[(\frac{n' -s'_i}{n'}) - ( 
\frac{n'-s'_i-t'_i+t_i}{n}) \Bigr]\, ,
\end{equation}
where $n'$ and $\delta'_0$ are the rank and degree of $E'$ respectively,
and if in \eqref{3} the equality holds, then 
\begin{equation}\label{4}
0\,\leq\, \Bigl[\frac{\delta_0}{n} - \frac{\delta'_0}{n'} 
\Bigr]\Bigl[-\frac{\delta_0}{n} - \frac{1}{2} + \sum_i( 
\frac{n'-s'_i-t'_i+t_i}{n})\Bigr]\, .
\end{equation}
\end{enumerate}
The pair $(E,\vec{g})$ is called {\it stable} if in addition one has strict 
inequalities in \eqref{1}, \eqref{2} and, when there is equality in \eqref{3}, in \eqref{4}.
\end{dfn}
The first two conditions arise as the stability criteria for a $\bbc^*$ action, and will be referred to as the {\it $\bbc^*$ stability conditions.}
The last two arise from the action of $Sl(V)$ on a space $V$, and will be referred to as the {\it $Sl(V)$ stability conditions.}
 
\subsection{The moduli construction}\label{moduli}

The moduli of pairs $(E,\vec{g})$, for which the planes correspond to 
framings was first examined by Seshadri in \cite{Se}, and considered 
more extensively and in a more general context by Huybrechts and Lehn 
in \cite{HL}. We adapt some of their notation and results to define a 
moduli space $\cagm\,=\, {\ca G\ca M}_{n,\delta_0, p_1,\cdots ,p_\ell}$, 
where $n$ is 
the rank. We begin by an essentially linear-algebraic construction which encodes 
the pair $(E,\vec{g})$ with $E$ of fixed rank $n$ and degree 
$\delta$. This follows a well-established pattern set, to 
name some, by Gieseker \cite{Gi}, Bhosle \cite{Bh1}, and Huybrechts and 
Lehn \cite{HL}. 

We shall first show that there exists an integer 
$e= e(n,g,\ell)$ such that for any semistable Grassmannian framed bundle $(E,\vec{g})$ 
of degree $\delta\ge e, \ E$ is generated by global sections and 
$H^1(X,E)=0$.

If $L$ is a line bundle on $X$, then fixing isomorphisms $(E\otimes L)_{p_i} \cong E_{p_i}$, we get a Grassmannian framed structure $\vec{g}_L$ on $E\otimes L$ induced from that on $E$. 
We call a Grassmannian framed bundle $(E,\vec{g})$ of rank $n$, degree $\delta$ 
\textit{pseudo semistable} if for every subbundle 
$E' \subset E$ of rank $n'$ and degree $\delta'$, we have 
\begin{equation}
0\quad \leq\quad \frac{\delta}{n} - \frac{\delta'}{n'} +\sum_i 
\Bigl[(\frac{n' -s'_i}{n'}) - ( \frac{n'-s'_i-t'_i+t_i}{n})\Bigr]\, .
\end{equation}
Then the Grassmannian framed bundle $(E\otimes L,\vec{g}_L)$ is pseudo 
semistable if and only if $(E,\vec{g})$ is pseudo semistable. A 
semistable Grassmannian framed bundle is pseudo semistable, but the 
converse may not be true.

\begin{lem}\label{bounded}
Let $(E,\vec{g})$ be pseudo semistable. If $\frac{\delta}{n} > 2g-2 + (n-1)\ell$, then 
$H^1(X,E)=0$. If $\frac{\delta}{n} > 2g-1 + (n-1)\ell$, then $E$ is generated by global sections. 
\end{lem}
\begin{proof}
Let $B_i= (\frac{n' -s'_i}{n'}) - ( \frac{n'-s'_i-t'_i+t_i}{n}), 
B= \sum_i B_i$. Then 
$$B_i= \frac{1}{nn'}(n'(n-n')-s'(n-n')-n'(t-t'))= \frac{(n-n')}{n} \frac{(n'-s')}{n'} -\frac{t-t'}{n} \le 1\, .$$ 
Hence $B \le \ell$.

If $H^1(E)\neq 0$, there exists a nonzero homomorphism 
$f: E \longrightarrow K$. Let $$\mu(E)\,=\, \frac{{\rm degree}\ E}{{\rm 
rank} \ E}\, .$$
Applying the pseudo semistability condition to the kernel of $f$ we have 
$$\mu(E) \le 2g-2 +(n-1)B \,\le\, 2g-2 +(n-1)\ell\, ,$$ i.e., $\mu(E) 
\le 
2g-2 + (n-1)\ell$. 
This contradicts $\frac{\delta}{n} > 2g-2 + (n-1)\ell$. 

For global generation of $E$ it suffices to have $H^1(X, E(-x))=0$ for all $x\in X$. 
Since $\mu(E(-x))= \mu(E)-1$, the result follows from the first part. 
\end{proof}

Let ${\mathcal O}_X(1)$ denote a fixed line bundle of degree $\ell$ over 
$X$. Fix a sufficiently large positive integer $k'$.
For a vector bundle $E_0$ on $X$, the Hilbert polynomial of 
$$E\,=\, E_0(k')\, :=\, E_0\otimes {\mathcal O}_X(k')$$ 
with $E_0$ of rank $n$ and degree $\delta_0$ is 
$$P_{k'}(t)\,=\, \chi(E_0({k'}+t))\,=\,nt+\delta_0+nk+n(1-g)\,=\,nt+ \delta+n(1-g)\, ,$$
where 
\begin{equation} k= k' \ell.\label{defk} \end{equation}
Set 
\begin{equation}p= P_{k'}(0)\,=\, \delta_0+kn+n(1-g)\, .\label{defp}\end{equation}

Let 
$$Quot\, := \,Quot({\mathcal O}_X^p, P_{k'}(t))$$ 
be the Quot scheme parametrizing all the quotients $q: {\mathcal O}_X^p 
\longrightarrow E$ such that
$E$ is a coherent sheaf on $X$ with Hilbert polynomial $P_{k'}(t)$. 
There exists a universal family ${\mathcal E}\longrightarrow 
Quot\times X$ and a 
(universal) quotient map 
$$ {\mathcal O}_{Quot\times X}^p\longrightarrow {\mathcal E}\, ,$$
such that for any $q\,\in\, Quot$, the restriction
${\mathcal O}_X^p\longrightarrow {\mathcal E}\vert_{\{q\}\times X}$ 
is represented by $q$.

 Let $$R\,\subset\, Quot$$ be the subset of $Quot$ 
consisting of
points $q\, \in\, Quot$ corresponding to sheaves ${\mathcal E}_q$ 
satisfying the following:
\begin{enumerate}
\item $E_q$ are vector bundles (generically) generated by sections, and
\item $H^0 (X, {\mathcal E}_q) \cong {\bbc}^{p}$
(so $H^1 (X,{\mathcal E}_q) = 0$ by the Riemann-Roch theorem).
\end{enumerate}
For sufficiently large $k$, the set $R$ contains the subset of
$Quot$ corresponding to all $E$ such that there is a Grassmannian 
framing $\vec{g}$ on $E$ satisfying the condition that the
pair $(E,\vec{g})$ is semistable (Lemma \ref{bounded}).
It is well known that $R$ is a Zariski open subset of $Quot$ \cite{Gi}.

 Let $p_R: R \times X \longrightarrow R$ be the
projection. Define 
$${\mathcal E}_{p_i}\,=\, (p_R)_{*} ({\mathcal E} \mid_{R \times p_i}) 
\,\longrightarrow \, R\, .$$ 
Let ${\bbc}^n_R$ be the trivial vector bundle of rank $n$ on $R$.
Let 
$${\rm Gr}_i({\mathcal E}_{p_i}\oplus {\bbc}^n_R) \longrightarrow R $$ 
be the Grassmannian bundles over $R$ 
whose fibers at $E$ are isomorphic to the Grassmannians of $n$-planes 
in $E_{p_i}\oplus{\mathbb C}^{n}$.
Let
$$
{\widetilde R}\, :=\, {\rm Gr}_1({\mathcal E}_{p_1}\oplus {\bbc}^n_R)
\times_R\cdots \times_R {\rm Gr}_{\ell}({\mathcal E}_{p_{\ell}}\oplus 
{\bbc}^n_R)\, \longrightarrow\, R
$$
be the fiber product. A point of $\widetilde R$ corresponds to a point 
$q$ of $R$, that is, a vector bundle $E$ and a point in the fiber of 
${\rm Gr}_i({\mathcal E}_{p_i}\oplus {\bbc}^n_R)$ for all $p_i$.

The group ${\rm GL}(p)\times {\rm GL}(n)^\ell$ acts on $Quot$ 
preserving $R$, and the action on $R$
lifts to $\widetilde{R}$.  Our moduli space $\cagm$ is 
a GIT-quotient of $\widetilde{R}$ by  ${\rm SL}(p)\times \bbc^*$, 
where $\bbc^*$ acts by multiples of the identity on the $\bbc^n$ and on the $\bbc^p$ factors. 
To construct the quotient, we use
an injective affine morphism of $\widetilde{R}$ into a suitable 
projective variety; this morphism will be described now.

Set $V \,=\, \bbc^p$. Since $\dim H^0(X,E)\,=\, p$, the 
vector bundle $E$ is then a 
quotient of the trivial vector bundle $V_X \,=\, V\otimes {\ca O}_X$ of 
rank $p$ on $X$. Fixing a quotient homomorphism
$V_X \, \twoheadrightarrow\,E$, we consider the determinant map
on sections: 
\begin{equation}\label{h}
h\, :\, \Lambda^n(V)\,\longrightarrow \, H^0(X,\, \det (E))\, .
\end{equation}

Let $\bbp({\ca U})$ be the projective Picard bundle over 
$\text{Pic}^d(X)\,=:\, A$, where $d\,=\, \delta_0+ nk$. We recall
that $\bbp({\ca U})$
parametrizes isomorphism classes of pairs consisting of a line bundle
of degree $d$ and a nonzero section of it.
As our determinant bundles $\det(E)$ lie in the above
component $A$ of the Picard group, the homomorphism $h$ in \eqref{h}
gives an element $\alpha$ of the projective bundle
$$
\alpha\, \in\, 
P := \bbp(Hom (\Lambda^n(V), {\ca U} ))\, \longrightarrow
\, A\, .
$$
This element $\alpha$ encodes the bundle $E$
\cite{Gi}, \cite{Bh1}, \cite{HL}. The 
Grassmannian framing $g_i$, in turn, defines under the evaluation map on 
sections of $E$ at $p_i$ a natural linear subspace in $V\oplus\bbc^n$ of
codimension $n$, and so 
$\vec{g}$ gets encoded as an $\ell$-tuple $\beta=(\beta_1,\cdots,\beta_{\ell})$ 
of elements $\beta_i$ of the Pl\"ucker embedding of the 
Grassmannian, meaning
$$\beta_i \,\in\,{\rm Gr}_n(V^*\oplus \bbc^{n}) \,\subset\, Q 
\,=\, \bbp(\Lambda^n(V^*\oplus\bbc^{n}))\, .$$ 
We note that the center of ${\rm GL}(p)$ acts non-trivially here on the $\beta_i$, even after projectivization; this is in contrast to many other moduli problems, such as those for parabolic bundles.

It is easy to see that associating the pair $(\alpha, \beta)$ to 
$(E,\vec{g})$ produces a morphism 
$$f: \widetilde{R} \longrightarrow P\times Q^{\ell}\, ,$$
lying over the morphism $f_R\,:\, R\,\longrightarrow\, P$ defined 
by $E \longmapsto \alpha$. 

 The set $f_R(R)$ of elements of $P$ is described in
\cite{Gi}, \cite{Bh1}, \cite{HL}; basically, under the evaluation at 
any point on the curve of the elements of $\ca U$, the element 
$\alpha_p\,\in\,\Lambda^n(V)^*$ that one gets must be a (non-zero) indecomposable 
element. Similarly, the elements $\beta_i$ must be indecomposable, 
meaning, they define an element of the Grassmannian of $p$-planes 
in 
$\bbc^{p+n}$ under 
the Pl\"ucker map. In addition, the elements $\alpha$ and $\beta_i$ must 
be 
compatible in the sense that the kernel of $\alpha_{p_i}$ must lie in 
the kernel of $\beta_i$. 

Let $Z$ be the Zariski closure of $f(\widetilde{R})$ in $P\times 
Q^{\ell}$. 

As usual, we need a polarization on $Z$. As in \cite{HL}, we obtain 
an ample line bundle ${\ca O}(1)_P$ on $P$ which is in the twist 
of the
lift of a very ample line bundle on $A$ by a line bundle that restricts 
to the standard positive generating bundle on each fiber 
(which is a projective space). We also have the standard ${\ca O}(1)_Q$ 
on $Q$. For a positive rational number $\eta\, =\, \nu/\mu$, where
$\nu$ and $\mu$ are integers, consider the polarization $\ca 
O(1)_P^{\otimes \mu} \boxtimes\left(\boxtimes_i {\ca 
O}(1)_Q^{\otimes\nu}\right)$.

As we have made a choice of a basis of $V$, we then quotient, taking the 
semi-stable elements. The quotient could be   by the action of  
${\rm GL}(V)$; we note however that the action of the center of ${\rm GL}(V)$ on the Grassmannians is equivalent to one by $\bbc^*$ acting with fixed weight $-\rho$ on $V^*$ and fixed but different weight $\sigma$ on $\bbc^n$, in the sense that their orbits are the same. On the other hand, they linearize differently; we will see that a quotient by an action of ${\rm SL}(V)\times \bbc^*$ on $(\alpha,\beta)$ is more appropriate. In general, to obtain correspondences between symplectic moduli and holomorphic moduli in other problems, e.g. in moduli of parabolic bundles, one must tune the polarizations on the holomorphic side. It is the case also here; but  here there is a supplementary tuning, in choosing the particular $\bbc^*$ action.

\subsection{Stability condition for points in $P\times Q^{\ell}$}

For the polarization corresponding to $\eta$, we want to examine the 
stability of the element $(\alpha,\beta) = 
(\alpha,(\beta_1,\cdots ,\beta_p)) $, first under the action of
${\rm SL}(V)$. 
We note that $(\alpha, \beta)$ is a (semi)stable point in $P\times Q^{\ell}$ if and only if 
it is a (semi)stable in $P' \times Q^{\ell}$ where 
$$P'\,=\, {\mathbb P}({\rm Hom} \ (\Lambda^n(V), H^0(\det (E)))^*)$$
with respect to the canonical linearization for the line bundle 
$\ca O(1)_{P'}^{\otimes \mu} \boxtimes (\boxtimes_i {\ca O}(1)_{Q}^{\nu})$ 
(see \cite[4.12]{Ma}, \cite[p.84]{HL}). 

We use the Hilbert criterion, as expounded in \cite{MFK}, which involves examining the action of all 
one-parameter subgroups of ${\rm SL}(p)$. This is equivalent to choosing 
a basis $v_i$ of $V^*$, and corresponding weights $a_i$ summing to zero, with $a_1\leq 
a_2\leq\cdots \leq a_n$, and taking the corresponding action. As remarked 
in \cite{HL}, the cone of these weights for the group $ {\rm 
SL}(V)\times \{\bbi\}$ acting on $V$ is generated by the weights 
$$((p'-p),(p'-p),\cdots ,(p'-p), p',\cdots ,p')\, ,$$
where the $(p'-p)$ is repeated $p'$ times and the $p'$ is repeated 
$(p-p')$ times. 

It suffices to consider stability for these generators. One now 
remarks that each of the set of choices (basis, generator of the 
cone of weights) corresponds to the choice of a $p'$-dimensional 
subspace $W$ of $V$ (the first $p'$ vectors) and a complementary 
space $W^\perp$ of it.

We consider the action corresponding to $(W\, , W^\perp)$ on 
$(\alpha\, , \beta)$. Decompose the representations in terms of weight 
spaces: let $x_i$ be a local basis of weight vectors for the action on 
the fibers of $P$, and set $\alpha= \sum_i\alpha_i x_i$; similarly, 
put $\beta_i=\sum_j\beta_{i,j} y_j$, for a basis of weight vectors 
$y_j$ for $Q$. Now define
$$
w_{W,\alpha} \,:=\, -\min_{\alpha_i\neq 0}{\rm weight}(x_i)~\, ~
\text{and}~\, ~ w_{W,\beta_i} \, :=\, -\min_{\beta_{i,j}\neq 0}{\rm 
weight}(y_j)\, .
$$
Setting 
$$w_W \equiv w_{W,\alpha} + \eta \sum_iw_{W,\beta_i}\, ,$$
for semistability (respectively, stability), one wants, as in \cite{HL}, that
$$0\,\leq\, w_W ~\,~\text{(respectively,}~\, 0\,< \, w_W{\rm )}$$
for all $p'$, $W$ and ${W}^\perp$.

\textbf{Notation.} We will use $(\leq)$ to denote $<$ for stability, 
and $\leq$ for semi-stability.

\begin{remark}
We will see that the choice of $W^\perp$ is irrelevant, 
and only $W$ counts; hence the notation $w_W$.
\end{remark}

Let $E_W$ be the subsheaf of $E$ generated by $W$. One has 
\cite[Lemma 1.23]{HL}, for $W$, with its accompanying weights:

\begin{lem}[\cite{HL}]
Let $n'\,=\, {\rm rank}(E_W)$. Then
$$w_{W,\alpha} \,=\,n'(p-p') - (n-n')p'\,=\,pn'-p'n\, .$$
\end{lem}

Given $\beta_i$, let $g_i\, \in\, {\rm Gr}_n(E_{p_i}\oplus 
\bbc^n)$ be the $n$-dimensional subspace that it defines. Let 
$$\Pi\,:\, 
E_{p_i}\oplus \bbc^n\,\longrightarrow\, E_{p_i}$$ 
be the natural projection. 
Let $E_{W,p_i}$ be the image of $E_W$ in the fiber of $E$ at $p_i$. We define 
$$m'_i \,=\, \dim(E_{W,p_i}), \quad s'_i \,=\, \dim (E_{W,p_i}\cap g_i), 
\quad t'_i \,= \,\dim (E_{W,p_i}/ (E_{W,p_i}\cap \Pi( g_i)))\, ,$$
$$
{\rm and}~\,~ r'_i \,=\, m'_i-s'_i-t'_i \,=\, \dim ( 
(E_{W,p_i}\cap \Pi(g_i))/(E_{W,p_i}\cap g_i))\, .$$
Note that if $E_W$ is a sub-bundle at $p_i$, then $m'_i=n'$; if the 
plane $g_i$ at $p_i$ is the graph of a framing, then $s'_i=t'_i = 
0$. Also, $s'_i\leq s_i, t'_i\leq t_i$, and $s'_i+t'_i\leq m'_i$. All of these quantities, when unprimed, refer to the case $W=V^*$.

The element $\beta_i$ can be written as 
$$\beta_i = b_1\wedge b_2\wedge\cdots\wedge b_{t_i}\wedge 
(e_{t_i+1}+b_{t_i+1})\wedge\cdots\wedge(e_n+b_n)\label{g-equation}\, .$$
Here $b_1,\cdots, b_{t_i}$ are independent elements of $V^*$, and
$e_{t_i+1},\cdots , 
e_{n}$ are independent elements of $(\bbc^n)^*$.
Now choose a basis $\{v_j\}$ of $V^*$ for which the first $p'$ vectors 
form a basis of $W$, and write the components of the $b_k$ as a matrix $b_{k,j}$.
Then row reduce (taking combinations of the $b_k$) and put the elements $b_1,\cdots , b_{t_i}$ in reduced 
echelon 
form with respect to this basis, permuting if necessary; let 
$c_1<\cdots <c_{t_i}$ be the indices for which $b_1,\cdots 
,b_{t_i}$ have nonzero coordinates in 
the basis for the first time. 

Use the $b_k$, $k= 1,\cdots ,t_i$, in their reduced echelon form to 
normalize 
the $c_j$-th entries, $j=1,\cdots ,t_i$ of $b_{t_i+1},\cdots , b_n$ to zero, 
and 
then 
put the $b_{t_i+1},\cdots ,b_n$ also into reduced echelon form, with 
$c_{t_i+1}, 
\cdots ,c_n$ the indices for which $b_{t_i+1},\cdots , b_n$ have a 
nonzero entry for the first time. 

\begin{lem} 
The following two hold:
\begin{align} t'_i & = \dim (E_{W,p_i}/ E_{W,p_i}\cap \Pi( g_i)) = {\rm \ number \ of }\ c_j\leq p'{\rm\ with}\ j\leq t_i \\
r'_i &= \dim (E_{W,p_i}\cap \Pi( g_i)/ E_{W,p_i}\cap g_i) = {\rm \ 
number \ of }\ c_j\leq p' {\rm\ with}\ j> t_i \, .\end{align} 
\end{lem}

\begin{proof}
The proof is fairly straightforward; one has
 $a\in \Pi(g_i) \iff b_j(a) = 0, j=1,\cdots ,t_i$, from which the first 
result follows. Also, $a \in g_i\cap E_{p_i} \iff b_j(a) = 0, j=1,\cdots ,n$, and 
so for all $i$,
$$m'_i-s'_i\,= \,\dim (E_{W,p_i}/ E_{W,p_i}\cap g_i) \,=\, {\rm \ 
number \ of }\ 
c_j\leq p'\, ,$$ from which the second result follows.
\end{proof} 

We note that ${\rm SL}(p)$ acts trivially on the $e_j$, and with weight 
$p'-p$ on $v_j, j=1,\cdots ,p'$, and with weight $p'$ on the rest. One 
now has the following result for (minus) the minimum weight:

\begin{lem} We have
$$w_{W,\beta_i} \,=\, t'_i(p-p') - (t_i-t'_i)p' + r'_i(p-p')\, 
=\, pt'_i-p't_i +r'_i(p-p')\, .$$
\end{lem}

Putting the results for $\alpha, \beta$ together,
we have:

\begin{prop}
The pair $(\alpha,\beta)$ is ${\rm SL}(V)$-semistable (respectively, 
stable) for the parameter $\eta$ if and only if for all planes $W$,
\begin{equation}
0\quad (\leq)\quad w_W \,=\, 
pn'-p'n +\eta \sum_i[pt'_i-p't_i + r'_i(p-p')]\, .
\label{W-inequality}\end{equation}
\end{prop}

\begin{lem}\label{compare}
Let $W$ and $W_1$ be two subspaces of $V$ such that each of $W, W_1$ generates the same subsheaf of $E$ and $W_1 \supset W$. Then 
$w_{W_1} \le w_W$ and if $W_1 \supsetneq W$ then $w_{W_1} < w_W$. 
\end{lem}
\begin{proof}
Let $p'=$ dim $W$, $p'_1= $ dim $W_1$. We have 
$$w_W - w_{W_1} = (p'_1 -p')(n + \eta \sum_i (t_i+r'_i)) \ge 0\, ,$$
with equality if and only if $p'= p'_1$.
\end{proof} 

We now choose our stability parameter, with $k$ as in \ref{defk}: 
\begin{equation}\label{sparam}
\eta \,=\, \frac{1}{k-g+\frac{1}{2}}.
\end{equation}

Set 
$$\delta\,=\, {\rm degree}(E).$$
 For any $W \subset V$, 
let 
$$\delta' \,=\, {\rm degree}(E_W),$$ where $E_W$ is the subsheaf of 
$E$ generated by $W$. 

Suppose that $H^1(X,E_W)\,=\,0, h^0(X, E_W)=p'$. 
Recalling \ref{defp}, one then has 
$$p\,= \,\delta +(1-g)n, \quad p' \,=\, \delta' + (1-g)n'\, .$$
Substituting into $w_W$, and dividing by $nn'$, one has the semistability condition for the action of $ {\rm SL}(V)$:
\begin{align}\label{sl-stab}0\quad (\leq)\quad \frac{\delta}{n}\Bigl(1+\eta\sum_i 
(\frac{r'_i +t'_i}{n'})\Bigr)& - \frac{\delta'}{n'}\Bigl(1+\eta \sum_i( \frac{r'_i+t_i}{n})\Bigr)\\ &+ 
\eta(1-g)\sum_i \Bigl[(\frac{r'_i +t'_i}{n'}) - ( 
\frac{r'_i+t_i}{n})\Bigr]\, .\nonumber\end{align}

As noted above, this vector bundle $E$ was twisted up from an original 
bundle $E_0$; one has 
$$ \delta \,=\, \delta_0 + kn,\,~ p \,=\, \delta_0+ (k-g+1) n$$
for some $k$. Likewise the subsheaf $E_{W}$ arises from a $E_{W,0}$ and 
$$ \delta' \,= \,\delta'_0 + kn',\,~ p' = \delta_0'+ (k-g+1) n'\, ,$$
where $\delta'_0$ is the degree of $E_{W,0}$.

Substituting into our expression \eqref{sl-stab}, we 
find
\begin{align}
0\quad (\leq) \quad w_{W,k} \quad = & \frac{\delta_0}{n} - \frac{\delta'_0}{n'} 
+ \sum_i \Bigl[(\frac{r'_i +t'_i}{n'}) - ( 
\frac{r'_i+t_i}{n})\Bigr]\\
+\, &\frac{1}{k}\sum_i\Bigl[\frac{\delta_0}{n}(\frac{r'_i +t'_i}{n'}) 
- \frac{\delta'_0}{n'}( \frac{r'_i+t_i}{n})
+(1-g)[(\frac{r'_i +t'_i}{n'}) - ( \frac{r'_i+t_i}{n})]
\Bigr]\nonumber\\ +\, &\frac{1}{k}[(\frac{1}{2}-g)(\frac{\delta_0}{n} - \frac{\delta'_0}{n'} )].
\nonumber
\end{align}
Let us refer to this condition as the $k$-(semi-)stability condition for 
the ${\rm SL}(V)$-action. Taking a limit, we have the 
$\infty$-(semi-)stability condition:
\begin{align}0\quad (\leq)\quad w_{W,\infty} \quad = & \frac{\delta_0}{n} - 
\frac{\delta'_0}{n'} + \sum_i \Bigl[(\frac{r'_i +t'_i}{n'}) - ( 
\frac{r'_i+t_i}{n})\Bigr]\, . \end{align}

\begin{prop}\label{h0h10}
Let
 $$w_{W,A}=\, \Bigl[\frac{\delta_0}{n} - \frac{\delta'_0}{n'} \Bigr]\Bigl[- \frac{\delta_0}{n}-\frac{1}{2}
 + \sum_i( \frac{r'_i+t_i}{n})\Bigr]\, .$$. \\
Then for 
$k> C$, where $C= C(n,l,g)$ is a constant, the following statements 
hold:
\begin{enumerate}
\item $w_{W,k}\ge 0$ if and only if $w_{W,\infty}\ge 0$ and in 
case $w_{W,\infty}=0$ one has 
$w_{W,A} \ge 0\, .$
\item $w_{W,k}>0$ if and only if $w_{W,\infty}\ge 0$ and in case 
$w_{W,\infty}=0$ one has 
$w_{W,A} > 0\, .$.
\item $w_{W,k}=0$ if and only if $w_{W,\infty}=0$ and then 
$w_{W,A}=0$.
\end{enumerate}
\end{prop}

\begin{proof}
 Set $B= \sum_i \Bigl[(\frac{r'_i +t'_i}{n'}) - (\frac{r'_i+t_i}{n})\Bigr].$

We shall first show that for $k > C_1$, the following holds: 
 $$(a) \ ~ \ {\rm if} \ w_{W,\infty}>0\, , \ {\rm then} \ w_{W,k}>0\, .$$
Substituting for the expression for $B$, the condition $W_{W,\infty} 
\,>\, 0$ takes the form $$\delta_0 n' -\delta'_0 n + n n' B>0\, .$$
Hence it implies that $\delta_0 n' -\delta'_0 n + n n' B\ge 1$ so that 
$- \frac{\delta'_0}{n'} \ge - \frac{\delta_0}{n}- B + \frac{1}{nn'}$. 
Substituting this in the expression for $w_{W,k}$ one gets 
$w_{W,k}\ge \frac{1}{nn'} -\frac{C_1}{knn'},$
where $C_1$ is a constant (i.e., dependent only on $n,\ell,g$). The 
claim (a) follows for $C=C_1$. 

Note that the statement 
$$(b') \ ~ \ {\rm if} \ w_{W,\infty}<0, \ {\rm then} \ w_{W,k}<0$$
is equivalent to the statement
$$(b) \ ~ \ {\rm if}\ w_{W,k}\ge 0, \ {\rm then} \ w_{W,\infty}\ge 0\, ;$$
 we prove $(b')$. 
The condition $w_{W,\infty} <0$ implies that $\delta_0 n' -\delta'_0 n + n n' B\le -1$. 
Hence $- \frac{\delta'_0}{n'} \ge - \frac{\delta_0}{n}- B - \frac{1}{nn'}$. 
Then $w_{W,k}\le \frac{-1}{nn'} +\frac{C_2}{knn'},$ 
where $C_2$ is a constant. Thus for $k > C_2$, we have $w_{W,k}<0$. 

Take $C$ to be the maximum of $C_1$ and $C_2$.

Suppose that $w_{W,k} \geq 0$, then by (b), 
$w_{W,\infty}\ge 0.$ Suppose that $w_{W,\infty}\,=\, 0$. 
Replacing $\frac{\delta'_0}{n'}$ by its value from the equality 
$w_{W,\infty}\,=\, 0$, the equation $w_{W,k}\ge 0$ becomes
$$0\,\leq\, \Bigl[\sum_i[(\frac{r'_i +t'_i}{n'}) - (
\frac{r'_i+t_i}{n}) ]\Bigr]\Bigl[ \frac{\delta_0}{n} + \frac{1}{2}
- \sum_i( \frac{r'_i+t_i}{n})\Bigr]\, ;$$
 equivalently, $w_{W,A}\ge 0$. 

Conversely, let $w_{W,\infty}\ge 0$ and in case of equality, $w_{W,A} 
(\ge) 0$. If $w_{W,\infty}> 0$, by (a) $w_{W,k}>0$. If $w_{W,\infty}= 
0$, then as seen above, 
$w_{W,k}\,\ge\, 0$ (respectively, $w_{W,k}\,>\, 0$) is equivalent to 
$w_{W,A}\,\ge\, 0$ (respectively, $w_{W,A}\,>\, 0$). 
This proves (1) and (2).

For (3), note that $(a)$ is equivalent to \\
$$(a') \ \ {\rm if} \ w_{W,k}\leq 0\, , \ {\rm then} \ w_{W,\infty}\leq 0\, .$$
It follows from $(a')$ and $(b)$ that for $k \ge C$, if $w_{W,k}=0$, 
then $w_{W,\infty}=0$ and then $w_{W,A}=0$. The converse can be proved 
as in (1) and (2).
\end{proof}

\begin{cor}\label{h10}
Let $W \subset V$ be such that $H^1(X, E_W)=0$. If $w_{H^0(X,E_W),k}\ge 
0$, then $W$ satisfies $w_{W}\ge 0$.
\end{cor}

\begin{proof}
We have $W \subset H^0(X,E_W)$ and both $W, H^0(X,E_W)$ generate $E_W$.
By Proposition \ref{h0h10} applied to $H^0(X,E_W)\subset V$ we have $w_{H^0(X,E_W)}= w_{H^0(X,E_W),k} \ge 0$ . By Lemma \ref{compare}, $w_{W} \ge w_{H^0(X,E_W)}\ge 0$.
\end{proof}

Let $K_X$ denote the canonical line bundle of $X$.

\begin{lem}\label{E'_W}
There exists a subsheaf $E'_W \subseteq E_W$ 
such that $E'_W$ is globally generated, 
$$H^1(X, E'_W)=0 \ {\rm and \ if} \ w_{H^0(E'_W)} \ge 0, \ {\rm then} \ w_{H^0(E_W)}\ge 0\, .$$
Thus, if $ E_W$ destabilizes, so does $E'_W$.
\end{lem}
\begin{proof}
If $H^1(X, E_W )= 0$, take $E'_W = E_W$.

If $H^1(X, E_W )\neq 0$, one has a map $E_W\longrightarrow K_X$, 
and an exact sequence
$$0\longrightarrow F \longrightarrow E_W\longrightarrow K_X$$
giving a rank $n'-1$ subbundle $F$ with $H^0(X,F)\geq (p'-g)$. Let 
$E^1_W$ be the subsheaf of $F$ generated by the global sections; then 
$H^0(X,E^1_W)\geq (p'-g)$. If $H^1(X, E^1_W)\neq 0$, 
we can then produce in a similar way a subsheaf $E^2_W \subset E^1_W$ of rank $n'-2$, with
$H^0(X,E^2_W)\geq (p'-2g)$. 
This process eventually terminates, \\
(1) either at a subsheaf $E^i_W$ with
$$H^0(X,E^i_W)\geq (p'-ig)~\, ~\text{and}~\, ~ H^1(X, E^i_W)= 
0$$
for some $i$, \\
(2) or at a line bundle $E^{n'-1}_W$ with 
$$H^0(X,E^{n'-1}_W)\geq 
(p'-(n'-1)g) ~\, ~\text{and}~\, ~ H^1(X, E^{n'-1}_W)= 0$$ 
if $p'\geq n'g$. \\
Call $E'_W$ the subsheaf at which the process terminates. Let $$W'= 
H^0(E_W), p'= h^0(E_W), W^{''}=H^0(X, E'_W), p^{''}=h^0(X,E'_W)\, .$$ 
Let 
$n',r',t'$ and $r'',t''$ be the corresponding quantities for $E_W$ and 
$E'_W$. 
The expression for $w_{W'}$ in (\ref{W-inequality}) can be 
rewritten as
$$w_{W'} \,=\, p(n'+\eta \sum_i(t'_i+r'_i)) - p'(n+ \eta 
\sum_i(r'_i+t_i))\, . $$
Using $p'\le p^{''}+mg, m\ge 1, \, n'=n''+m$, this gives 
$$w_{W'} \, \ge p(n''+m+ \eta \sum_i(t'_i+r'_i)) - (p''+mg)(n+ \eta 
\sum_i(r'_i+t_i))\, .$$ 
Hence we get
$$w_{W'}- w_{W^{''}} \, \ge\, 
pm-mgn+\eta\Bigl[\sum_i(p(t'_i+r'_i-t''_i-r''_i)-p''(r'_i-r''_i)-mg(t_i+r''_i)\Bigr]\, .$$
If $w_{W^{''}}\ge 0$, then $p'' \,\le\, \frac{pn}{n''}+ C'$, with $C'$ a 
constant. 
Substituting this in the expression for $w_{W'}- w_{W^{''}}$, one sees that for $k$ large ($k \ge$ a constant), 
 $\eta$ is small enough so that $w_{W'}- w_{W^{''}}\ge 0$.
\end{proof}

\begin{prop}\label{rightiml}
If every subsheaf $E'$ of $E$ satisfies the conditions 
$$
0\quad \leq\quad S^1(E'):= \frac{\delta_0}{n} - 
\frac{\delta'_0}{n'} + [\sum_i[(\frac{r'_i +t'_i}{n'}) - (
\frac{r'_i+t_i}{n}) ]\, ,
$$
and if one has equality, 
 $$0\,\leq\, S^2(E'):= \, \Bigl[\frac{\delta_0}{n} - 
\frac{\delta'_0}{n'} \Bigr]\Bigl[- \frac{\delta_0}{n}-\frac{1}{2}
 + \sum_i( \frac{r'_i+t_i}{n})\Bigr]\, ,
 $$
 then $(\alpha, \beta)$ is k-${\rm SL}(V)$-semistable.

If every subsheaf $E'$ of $E$ satisfies the conditions 
$$S^1(E') \ge 0\, , \ {\rm and \ for } \ S^1(E')=0 \ {\rm one \ has} \ S^2(E')>0\, ,$$
then $(\alpha, \beta)$ is k-${\rm SL}(V)$-stable.
\end{prop}

\begin{proof}
Suppose that every subsheaf of $E$ satisfies the conditions in the statement of Proposition \ref{rightiml}. 
Let $W \subset V$ be a subspace and $E_W$ the subsheaf of $E$ generated by $W$. 
By Lemma \ref{E'_W}, there exists a subsheaf $E'_W \subseteq E_W, W''=H^0(X,E'_W)$ 
such that $E'_W$ is globally generated, 
$H^1(X, E'_W)=0 \ {\rm and \ if} \ w_{W''} \ge 0, \ {\rm then} \ w_{H^0(E_W)}\ge 0\, .$
One has $w_{W'', \infty}=S^1(E'_W), w_{W'',A}=S^2(E'_W)$. 
 Since $E'_W \subset E$ satisfies the conditions in the statement of Proposition \ref{rightiml},
by Proposition \ref{h0h10}(1), $w_{W''}= w_{W'',k} \ (\ge)\, 0$. 
Hence by Lemma \ref{E'_W}, $w_{H^0(E_W)}\, (\ge)\, 0$. 
 By Lemma \ref{compare}, $w_{W} \ge w_{H^0(X,E_W)}\, (\ge)\, 0$. 
Thus $(\alpha, \beta)$ is $k$-${\rm SL}(V)$-(semi)stable.
\end{proof}

To prove the converse of Proposition \ref{rightiml}, 
we need the following lemma.

\begin{lem}\label{leftlem}
Suppose that there is a subsheaf $F$ of $E$ satisfying the conditions 
$S^1(F) \le 0$ and if $S^1(F)=0$, then $S^2(F) < 0$ (respectively, 
$S^2(F)\,=\, 0$). 
Then for degree $E$ large, there exists a subsheaf $E' \subset E$ 
satisfying the same respective conditions and with $H^1(X,E')=0$. 
\end{lem}
\begin{proof} 
Suppose that $E$ has a subsheaf $F$ satisfying the conditions 
$S^1(F) \le 0$ and if $S^1(F)=0$, then $S^2(F) < 0$. 
Take $E'\subset E$ be a subsheaf of rank $n'$ and degree $\delta'$ such that 
\begin{equation}\label{min}
S^1(E') = \ {\rm min} \{ S^1(F) \mid S^1(F)\le 0 \ {\rm and \  if} \ S^1(F)=0, \ {\rm then} \  S^2(F) <0\, \}\, . 
\end{equation}
Suppose that $H^0(X, E'^*\otimes K_X)\neq 0$, i.e., there is a nonzero 
homomorphism 
$$\phi: E' \longrightarrow K_X\, .$$ Let $E''\,=\, {\rm Ker} (\phi)$, 
and let $n'', \delta''$ be its rank and degree respectively. 
By the choice of $E', \ S^1(E') \le S^1(E'')$ which gives 
\begin{equation}\label{ii}
\frac{\delta''}{n''} \le \frac{\delta'}{n'} + D \, ,\end{equation}
where
$$D:= \sum_i[(\frac{r''_i +t''_i}{n'-1}) - (\frac{r'_i+t'_i}{n'}) + (\frac{r'_i-r''_i}{n})]\, .$$

One has $n''= n'-1, \ \delta'' \ge \delta' -2g+2\, .$ Hence 
\begin{equation}\label{i}
\frac{\delta''}{n''} \, \ge \, \frac{\delta'}{n'-1} +\frac{2-2g}{n'-1} 
\,=\, \frac{\delta'}{n'(n'-1)} + \frac{2-2g}{n'-1}-D + 
\frac{\delta'}{n'} + D\, . \\
\end{equation}

Since $S^1(E')\le 0$, we have 
$$\frac{\delta'}{n'} \,\ge\,
\frac{\delta}{n} + \sum_i[(\frac{r'_i +t'_i}{n'}) - 
(\frac{r'_i+t_i}{n}) ]\, .
$$
Hence for $\delta$ larger than a constant, we have 
$$\frac{\delta'}{n'(n'-1)} + \frac{2-2g}{n'-1}-D >0\, .$$
Then $\frac{\delta''}{n''} \,>\, \frac{\delta'}{n'} + D\, ,$ 
contradicting the inequality \eqref{ii}. This proves the lemma.

In case $S^1(E') \,\le\, 0$, and for $S^1(E')\,=\,0$ 
one has $S^2(E')\,=\,0$, 
 we only need to change $S^2 F <0$ to $S^2(F)=0$ in \eqref{min} 
in the choice of $E'$. 
\end{proof}

\begin{prop}\label{leftiml}
The point $(\alpha, \beta)$ is k-${\rm SL}(V)$-(semi)stable for $k\ge 
k_0(n,g,\ell)$ 
if and only if every subsheaf $E'$ of $E$ satisfies the conditions 
$$S^1(E') \ge 0\, , \ {\rm and \ if } \ S^1(E')=0, \ {\rm then} \ S^2(E') 
\,(\ge)\, 0\, .$$
\end{prop}

\begin{proof}
In view of Proposition \ref{rightiml}, it remains to show that 
if $(\alpha, \beta)$ is ${\rm SL}(V)$-(semi)stable, then 
every subsheaf $E'$ of $E$ satisfies the conditions 
$$S^1(E') \ge 0\, , \ {\rm and \ if } \ S^1(E')=0, \ {\rm then} \ 
S^2(E')\, (\ge)\, 0\, .$$

Suppose that there is a subsheaf $F \subset E$ such that 
$S^1(F) \le 0$ and if $S^1(F)=0$, then $S^2(F) <0$ (respectively, 
$S^2(F)=0$). By Lemma \ref{leftlem}, 
there is a subsheaf $E' \subset E$ satisfying the same 
respective conditions and with $H^1(X,E')=0$. 
Then for $W= H^0(X,E')$, one has $w_{W,\infty}=S^1(E')$ and $w_{W,A}= 
S^2(E')$. By Proposition 
\ref{h0h10}, this implies that $w_{W}<0$ (respectively, $w_{W}=0$) 
contradicting 
the ${\rm SL}(V)$-semistability (respectively, stability) of $(\alpha, 
\beta)$. 
\end{proof}

We would like to have conditions in Proposition \ref{leftiml} to be converted into conditions for subbundles of $E$. Let $E' \subset E$ be a subsheaf and $E^c$ the minimal subbundle of $E$ containing $E'$. Define a subsheaf $\widehat{E} \subset E$ by 
\begin{equation}\label{ehat}
0 \longrightarrow \widehat{E} \longrightarrow E^c \longrightarrow 
\oplus_i E^c_{p_i}/(E'_{p_i}+(E^c_{p_i}\cap g_i)) \longrightarrow 
0\, .
\end{equation}
Then $\widehat{E}$ satisfies the following conditions:
 
1) $\widehat{E} = E^c$ away from $p_i$, 

2) $\widehat{E}_{p_i}\cap g_i= E^{c}_{p_i}\cap g_i$, 

and since $E'_{p_i}\cap (E^c_{p_i}\cap g_i)=E'_{p_i}\cap g_i$, 

3) $E'_{p_i}/(E'_{p_i}\cap g_i) \cong \widehat{E}_{p_i}/\widehat{E}_{p_i}\cap g_i$.

The condition 2) says that $s^c_i= \widehat{s}_i$ and the condition 3) 
gives $m'_i- s'_i
\,=\, \widehat{m}_i-\widehat{s}_i$ or equivalently,
$$s^c_i\,=\, \widehat{s}_i\, , ~ r'_i+ t'_i \,=\, \widehat{r}_i+ 
\widehat{t}_i\, .$$ 
Moreover, $s^c_i= \widehat{s}_i$ and $E^c_{p_i}\cap \Pi g_i \supseteq 
E'_{p_i}\cap \Pi g_i$, 
hence $r^c_i \ge \widehat{r}_i$ with equality holding if and only if 
$E^c_{p_i}\cap \Pi g_i = E'_{p_i}\cap \Pi g_i$.

We have
$$S^1(E') -S^1(\widehat{E})\,=\, \frac{\widehat{\delta}_0 - 
{\delta}'_0}{n'}
+ \sum_i \frac{\widehat{r}_i-r'_i}{n}\, .$$ 
{}From the defining sequences of $E'$ and $\widehat{E}$, it follows that 
$$\widehat{\delta}_0 - {\delta}'_0\,= \,\sum_i \ \dim
\widehat{E}_{p_i}/E'_{p_i}
+ \delta(T)\, ,$$ where $T$ is a torsion sheaf supported outside the 
$\{p_1\, ,\cdots\, ,p_\ell\}$. 
Hence $$\widehat{\delta}_0 - {\delta}'_0\,=\, \sum_i [(m'_i 
+s^c_i-s'_i)-m'_i]+\delta(T)
\,=\, \sum_i (s^c_i-s'_i) + \delta(T)\, .$$
Therefore,
$$S^1(E') -S^1(\widehat{E})= \sum_i \frac{(s^c_i-s'_i)}{n'} 
+ \sum_i \frac{\widehat{r}_i-r'_i}{n}+ \frac{\delta(T)}{n'}
\ge \sum_i (\frac{(\widehat{s}_i +\widehat{r}_i)- (s'_i+r'_i)}{n}+ \frac{\delta(T)}{n'}$$
with equality holding if and only if $n=n'$.
Now, $$\widehat{s}_i +\widehat{r}_i\,=\, \dim 
(\widehat{E}_{p_i}\cap \Pi g_i)\, , 
s'_i+r'_i\,= \,\ \dim ({E'}_{p_i}\cap \Pi g_i)\, ,$$ hence
$$
S^1(E') \,\ge\, S^1(\widehat{E}) + \delta(T)/n'$$
with equality holding if and only if $n=n'$ and 
$\widehat{E}_{p_i}\cap \Pi g_i= {E'}_{p_i}\cap \Pi g_i$ for all $i$.

We now compare $S^1(\widehat{E})$ and $S^1(E^c)$. From the sequence (\ref{ehat}), 
$$\widehat{\delta}_0 \,=\, \delta^c_0 + \sum_i 
[(r'_i+t'_i)-(r^c_i+t^c_i)]\, .$$ 
Substituting for $\widehat{\delta}_0/n'$ in $S^1(\widehat{E})$, we have
$$S^1(\widehat{E})\,=\, \frac{\delta_0}{n}- 
\frac{\delta_0^c}{n'}+ 
\sum_i [\frac{(r^c_i+t^c_i)}{n'}- \frac{\widehat{r}_i+t_i}{n}]
= S^1(E^c) + \sum_i \frac{r^c_i-\widehat{r}_i}{n}
\ge S^1(E^c)\, ,$$
with equality holding if and only if $r^c_i= \widehat{r}_i$ for 
all $i$. Hence
$$S^1(E') \,\ge\, S^1(\widehat{E}) +\delta(T)/n' \ge S^1(E^c) + 
\delta(T)/n'$$
with equality holding if and only if $n=n'$, 
$r^c_i= \widehat{r}_i$ and $\widehat{E}_{p_i}\cap \Pi g_i= 
{E'}_{p_i}\cap \Pi g_i$ for all $i$.

Thus we have proved the following lemma.

\begin{lem}\label{scompare} 
We have
\begin{equation}
S^1(E') \,\ge\, S^1(\widehat{E}) +\delta(T)/n' \,\ge\, S^1(E^c) + 
\delta(T)/n'
\end{equation}
with $S^1(E')= S^1(E^c)$ if and only if $n=n', T=0$ and 
${E^c}_{p_i}\cap \Pi g_i= \widehat{E}_{p_i}\cap \Pi g_i= {E'}_{p_i}\cap \Pi g_i$ for all $i$.
\end{lem}

\begin{thm}\label{s2s}
There exists $k_0\,=\,k_0(n,g,\ell)$ such that for all $k\ge k_0$, 
the point $(\alpha, \beta)$ given by $(E,\vec g)$ is 
k-${\rm SL}(V)$-(semi)stable if and only if for all subbundles $E'$ of 
$E$, 
$$
0\quad \leq\quad \frac{\delta_0}{n} - 
\frac{\delta'_0}{n'} + \sum_i\Bigl[(\frac{n' -s'_i}{n'}) - ( 
\frac{n'-s'_i-t'_i+t_i}{n}) 
\Bigr]\, ,
$$
and, for any $E'$ for which one has equality, 
 $$0\,\leq\, \Bigl[\frac{\delta_0}{n} - \frac{\delta'_0}{n'} \Bigr]\Bigl[- \frac{\delta_0}{n}
- \frac{1}{2} + \sum_i( \frac{n'-s'_i-t'_i+t_i}{n})\Bigr]\, ,$$
 and $$\sum_i s_i\, \ (\leq)\, \ \frac{n(2\ell-1)}{2} - \delta_0\, .$$
\end{thm}

\begin{proof}
Suppose that $S^1(E^c)\ge 0$ for all subbundles $E^c \subsetneq E$. 
Let $E'$ be a subsheaf of $E$. By Lemma \ref{scompare}, $S^1(E') \ge 
0$ and $S^1(E')=0$ if 
and only if $n=n', T=0$ and 
${E^c}_{p_i}\cap \Pi g_i= \widehat{E}_{p_i}\cap \Pi g_i= {E'}_{p_i}\cap \Pi g_i$ for all $i$. 
Hence if $S^1(E')=0$, then $$r^c_i= \widehat{r}_i=r'_i, s^c_i= 
\widehat{s}_i=s'_i\, .$$
 For $n=n'$, $E^c=E, r^c_i +t_i= n-s_i$. Then 
 $$\sum_i(\frac{r'_i+t_i}{n})= \sum_i(\frac{n-s_i}{n})\, .$$ 
Therefore 
$$S^2(E')= (\frac{{\delta}_0 - {\delta}'_0}{n})[-\frac{1}{2}
- \frac{\delta_0}{n}\sum_i\frac{(r'_i+t_i)}{n}] 
= (\frac{{\delta}_0 - {\delta}'_0}{n})[\sum_i\frac{(n-s_i)}{n}-\frac{1}{2}
- \frac{\delta_0}{n}]\, . $$
Since $(\delta_0 - \delta'_0)/{n}> 0$, 
$$S^2(E')\, \ (\ge)\, \ 0 $$ 
 if and only \ if 
$$-\frac{1}{2}-\frac{\delta_0}{n}+  \sum_i\frac{(n-s_i)}{n} \, \ (\ge)\, \ 0$$ 
i.e., 
$$\sum_i s_i\, \ (\leq)\, \ \frac{n(2\ell-1)}{2} - \delta_0\, .$$
The theorem now follows from Proposition \ref{rightiml} taking $k>>0$ 
$(k \ge k_0(n,g,\ell))$ and noting that $ n' \,=\,s'_i+ r'_i + 
t'_i$ for the subbundles.
\end{proof}

We now turn our attention to the action of $\bbc^*$.
We will take it to act on $V^*$ with weights $-\rho$, and on $\bbc^n$ with weight $\sigma$.
The action on $\alpha$ then has a fixed weight $-n\rho$; the action on each $\beta_i$ has lowest weight $-t_i\rho-r_i\rho+s_i\sigma= (s_i-n)\rho +s_i\sigma$, and highest weight $-t_i\rho +r_i\sigma +s_i\sigma = -t_i \rho+  (n-t_i)\sigma$. For (semi)-stability, the highest and lowest weight must bracket the origin, and this gives the (semi-)stability condition for the action of $\bbc^*$ (the constant $\eta$ is as in \ref{sparam}):
\begin{align}\label{c*-stab}0\quad (\leq)&\quad n\rho + \eta\sum_i ( (n-s_i)\rho -  s_i\sigma)\\
0\quad (\leq)&\quad -n\rho + \eta\sum_i (- t_i\rho +  (n-t_i)\sigma)\, .
 \nonumber\end{align}

We set 
\begin{equation}
\rho(2n(k-g+\frac{1}{2}+\ell) +2\delta_0) = \sigma(\ell n-2\delta_0).
\label{C-stab-param}\end{equation}
Substituting, we have

\begin{thm}\label{s1s}
Let the weights $(\rho,\sigma)$  be as in 
\eqref{C-stab-param}. Let $k$ be sufficiently large. 
Then $(E,\vec g)$ is $k$-$\bbc^*$ (semi)stable if and only if
\begin{align}
 \sum_{i=1}^{\ell}s_i\quad (\leq) \quad & \frac{\ell n}{2} - 
\delta_0\, , \label{1-s-bound}\\
 \sum_{i=1}^{\ell}t_i\quad (\leq) \quad & \frac{\ell n}{2} + 
\delta_0\, .\label{1-t-bound}\end{align}
\end{thm}

\subsection {The moduli space} 

\begin{thm}
There exists a projective scheme ${\cagm} \,=\, {\ca G \ca 
M}_{n,\delta_0,p_1,\cdots ,p_\ell}$ which is a coarse moduli space for 
semistable Grassmannian framed bundles
of rank $n$ and fixed degree with Grassmannian framed structures at 
$p_1,\cdots ,p_{\ell}$.
\end{thm}

\begin{proof}
In Section \ref{moduli}, we defined an ${\rm SL}(p)\times 
\bbc^*$-equivariant morphism 
$$f: {\widetilde R} \longrightarrow Z \subset P\times Q^{\ell}\, 
.$$
Let ${\widetilde R}^{ss}$ denote the points corresponding to semistable 
Grassmannian framed bundles, and let $(P\times Q^{\ell})^{ss}$ 
denote the 
semistable points for $ {\rm SL}(p)\times \bbc^* $-action. From Theorem 
\ref{s2s} and Theorem \ref{s1s}, it follows that 
$f$ induces a morphism 
$$f^{ss}: {\widetilde R}^{ss} \longrightarrow Z^{ss} \, .$$ 
In fact, 
$$Z \subset (P\times \text{Gr}_n(V^*\oplus \bbc^{n})^{\ell})^{ss} 
\,\subset \, (P\times Q^{\ell})^{ss}\, .$$
Using the properness of ${\rm Gr}_n(V^*\oplus \bbc^{n}) \subset 
Q$, as in \cite[Proposition 3]{Bh1}, we can prove the valuative 
criterion of 
properness for the morphism $f^{ss}$. Thus $f^{ss}$ is proper. It is 
also injective and hence affine. Therefore, the existence of the 
quotient 
of the projective scheme $Z^{ss}$ by $ {\rm SL}(p)\times \bbc^*$ implies the existence of 
the projective scheme ${\cagm}= {\widetilde R}/ ({\rm SL}(p)\times 
\bbc^*)$, the GIT-quotient of ${\widetilde R}$ by $ {\rm SL}(p)\times 
\bbc^*$.
\end{proof}

\subsection{Relation to parabolic structures}

The canonical basis $e_1,\cdots,e_n$ of $\bbc^n$ 
defines a natural flag of subspaces 
$\bbc^1\subset\bbc^2\subset\cdots\subset\bbc^n$. 
Consider a pair $(E,\vec g)$. 
The direct sum $E_{p_i}\oplus \bbc^n$ has projections $\Pi$ and $R$ to 
$E_{p_i}$ and $\bbc^n$ respectively. One has a flag 
$$\{0\}\subset R^{-1}(\bbc^1)\subset R^{-1}(\bbc^2)\subset\cdots\subset 
R^{-1}(\bbc^n)$$ in $E_{p_i}\oplus \bbc^n$. The plane $g_i$ intersects 
this flag, and one can project the intersections, using $\Pi$, to $E_{p_i}$, giving a nested sequence of subspaces
$$ 0= F_{i, -1} \subset F_{i,0}= (E_{p_i}\cap g_i)\subset 
F_{i,1}\subset F_{i,2}\subset \cdots\subset F_{i,n}= 
(E_{p_i}\cap\Pi(g_i))\subset F_{i, n+1}= E_{p_i}\, .$$
Note that for a subbundle $E'$ of $E$, one has an induced flag $F'_i$ in 
$E'_{p_i}$.

For convenience, parabolic weights will take values in 
the interval $[-1/2, 1/2]$ instead of $[0,1)$ (as we will be relating these to a moment map taking values in the interval $[-1/2, 1/2]$).
Now choose the weight $\alpha_{i,0} = 1/2$ for $F_{i,0}$, weight $\alpha_{i,n+1} 
=-1/2$ for $E_{p_i}/F_{i,n}$ and weights $\alpha_{i,j}$ for $F_{i,j}/F_{i,j-1}$, with 
$1/2 >\alpha_{i,1}\geq \alpha_{i,2}\geq\cdots\geq\alpha_{i,n}>-1/2$.
We have $$s_i = \dim(F_{i,0})= 
\text{multiplicity~of~the~weight~}
1/2\, ,$$ 
$$t_i= \dim (F_{i,n+1}/F_{i,n})\,=\,  
 \text{multiplicity~of~the~weight~} -1/2,$$
  and similarly for subbundles
$E'$.

Define as usual the parabolic degree to be
$${\rm pardeg} (E) \,=\, \delta_0(E) + \sum_{i=0}^{\ell} 
\sum_{j=0}^{n+1} \dim (F_{i,j}/F_{i, j-1})\alpha_{i,j}\, .$$
The usual definition of parabolic (semi)stability applies, in that one 
asks that for a subbundle $E'$ of rank $n'< n$, 
$$0\quad (\leq)\quad \frac{{\rm pardeg} (E)}{n} -\frac{{\rm pardeg} 
(E')}{n'}\, .$$

\begin{prop}\label{propp}
Let $(E,\vec g)$ satisfy the $\bbc^*$ (semi)stability conditions \ref{1}, \ref{2}, 
and let it be equipped with compatible flags in $E_{p_i}$ as 
above; if the result is parabolic (semi)stable for any  one  choice $\underline{\alpha}$ of 
weights $$\alpha_{i,0}\,=\, 1/2 \,>\,\alpha_{i,1}\,\geq\, 
\alpha_{i,2}\,\geq\,\cdots\,\geq\, \alpha_{i,n}\,>\,-1/2 \,=\, 
\alpha_{i,n+1}$$
as above, then it is (semi)stable as a Grassmannian framed bundle.
\end{prop}

\begin{proof}
One has, for any $E'$, the condition for $\underline{\alpha}$-parabolic (semi)stability,
$$0\,(\leq)\,\frac{\delta(E)}{n} -\frac{\delta(E')}{n'} + \sum_{i=1}^\ell 
\Big[\frac{1}{n}\sum_{j=0}^{n+1} \dim (F_{i,j}/F_{i, j-1})\alpha_{i,j}-
\frac{1}{n'}\sum_{j=0}^{n+1} \dim (F'_{i,j}/F'_{i, j-1})\alpha_{i,j} 
\Big]\, .$$

The previous inequality becomes:
\begin{align} 0\,(\leq)\, \frac{\delta(E)}{n} -\frac{\delta(E')}{n'} +& 
\sum_{i=1}^\ell \Big[\frac{1}{n}\sum_{j=1}^{n} \dim (F_{i,j}/F_{i, j-1})\alpha_{i,j} 
- \frac{1}{n'}\sum_{j=1}^{n} \dim (F'_{i,j}/F'_{i, 
j-1})\alpha_{i,j}\nonumber\\ + & \frac{1}{2n}(s_i-t_i) - 
\frac{1}{2n'}(s'_i-t'_i)\nonumber\Big]\, .\end{align}

We would like this to imply Grassmann-framed semistability, for any 
choice of $\alpha_{i,j}$ within our simplex of weights. The inequality is an affine one 
in the $\alpha_{i,j}$, and so it suffices to check this for the vertices of 
the simplex; this corresponds to considering choices of weights of the 
form $$\alpha_{i,0}\,=\,\cdots\,=\,\alpha_{i,k_i} \,=\, 1/2\, ,
\alpha_{i,k_i+1}\,=\,\cdots\,=\,\alpha_{i,n+1}\,=\, -1/2\, .$$
For the vector bundles $E$ and $E'$, let $m_{k_i}:= \dim 
(F_{i,k_i}/F_{i,0})$ and
$m'_{k_i}:= \dim (F'_{i,k_i}/F'_{i,0})$. The inequality becomes:
$$0\,(\leq)\, \frac{\delta(E)}{n} -\frac{\delta(E')}{n'} + 
\sum_{i=1}^\ell 
\Big[\frac{1}{2n}(m_{k_i} - (r_i-m_{k_i}) + s_i-t_i) - 
\frac{1}{2n'}(m'_{k_i} - (r'_i-m'_{k_i}) +s'_i-t'_i)\Big]\, .$$
The right hand side is equal to 
$$
\frac{\delta(E)}{n} - \frac{\delta(E')}{n'} +
\sum_{i=1}^\ell \Big[(\frac{n' -s'_i}{n'}) - ( 
\frac{n'-s'_i-t'_i+t_i}{n})\Big]
$$
$$
+\sum_{i=1}^\ell \Big[\frac{1}{2n}(2m_{k_i} -r_i + s_i-t_i +2r'_i 
+2t_i) - \frac{1}{2n'}(2m'_{k_i} - r'_i 
+s'_i-t'_i+2n'-2s'_i)\Big]\, .
$$
One then needs
\begin{equation}\label{s3s}
\sum_{i=1}^\ell \Big[\frac{1}{2n}(2m_{k_i} -r_i + s_i-t_i +2r'_i 
+2t_i) - \frac{1}{2n'}(2m'_{k_i} - r'_i +s'_i-t'_i+2n'-2s'_i)\Big]\leq 
0\, .
\end{equation}

To prove \eqref{s3s}, note that
\begin{align}& \sum_{i=1}^\ell \Big[\frac{1}{2n}(2m_{k_i}-r_i + s_i-t_i +2r'_i +2t_i) - \frac{1}{2n'}(2m'_{k_i}- r'_i +s'_i-t'_i+2n'-2s'_i)\Big]\nonumber\\ 
=\, & \sum_{i=1}^\ell \Big[\frac{1}{2n}(2m_{k_i}- 2r_i + n +2r'_i) - 
\frac{1}{2n'}(2m'_{k_i} +n')\Big]\nonumber\\
=\, &\sum_{i=1}^\ell \Big[ \frac{1}{2n}(2m_{k_i} - 2r_i+2r'_i) - 
\frac{1}{2n'}(2m'_{k_i} )\Big]\nonumber\\
\leq\, &\sum_{i=1}^\ell \Big[\frac{1}{n}(m_{k_i} - r_i-m'_{k_i} 
+r'_i)\Big] \nonumber \end{align}
which is indeed less or equal to zero, since
$(r_i-m_{k_i})\, =\,\dim (F_{i,n}/F_{i,k_i})$ and $(r'_i-m'_{k_i})\,=\, 
\dim (F'_{i,n}/F'_{i,k_i})$.
\end{proof}

When $s_i, t_i =0$, the flag $F_{i,j}$ constructed above is  a full flag, with $j$ corresponding to the  dimension. On the other hand, if $s_i$, $t_i$ are arbitrary, then one has the sequence of spaces $F_{i,j}$ interpolating between $E_{p_i}\cap g_i$, of dimension $s_i$, and  
$\Pi(g_i)$, of codimension $t_i$, both subspaces of $E_{p_i}$.  The dimensions of the $F_{i,j}$ do not necessarily follow a regular pattern, varying according to how $\bbc^n\cap g_i$ intersects the standard flag $\bbc^j$ (generated by the first $j$ vectors of the standard basis) in $\bbc^n$. We now suppose that this intersection is maximal (this will correspond to taking a closed orbit in the quotient construction to come):
\begin{equation}\label{closed}\bbc^n\cap g_i= \bbc^{t_i}.\end{equation}
 This fixes the dimensions of $R_i^{-1}(\bbc^{t_i+s}) $ to $t_i+s+s_i$  and that of $F_{i, t_i+s}$ to $s+s_i$ for $s\leq n-s_i-t_i$,. We thus get a sequence of spaces
$$ 0  \subset F_{i,0}= (E_{p_i}\cap g_i)=\cdots =
F_{i,t_i}\subsetneq F_{i,t_i+1}\subsetneq \cdots  \subsetneq F_{i,n-s_i} =
\cdots = F_{i,n}= 
\Pi(g_i)\subset  E_{p_i}\, .$$
with a sequence of dimensions
$$0\leq s_i=\cdots =s_i < s_i+1 < \cdots < n-t_i= \cdots =n-t_i\leq n$$

Now consider the sequence of weights
$$\alpha_{i,0} =1/2 >\alpha_{i,1}\geq \alpha_{i,2}\geq\cdots\geq\alpha_{i,n}>\alpha_{i,n+1} =-1/2.$$
 We  note that the weights $\alpha_{i,1},\cdots ,\alpha_{i,t_i}$ are irrelevant
for $t_i>0$, as the dimensions of the corresponding vector spaces in the flag are
constant in this range of indices; likewise for the weights
$\alpha_{i,n-s_i+1},\cdots ,\alpha_{i,n}$ if $s_i>0$. The weights $\alpha_{i,t_i+1},\cdots ,\alpha_{i,n-s_i}$ also allow repetitions. Set
$\alpha_{i,t_i+1}=\cdots =\alpha_{i,t_i+j_i^1}< \alpha_{i,t_i+j_i^1+1}=\cdots =\alpha_{i,t_i+j_i^2}<\cdots <\alpha_{i,t_i+ j_i^{k_i-1} +1}=\cdots =\alpha_{i,j_i^{k_i}}$ 
with $j_i^{k_i}= n-s_i$; there are thus $k_i$ different weights  $\alpha_{i,j}$ in the sequence, as well as the weights $\alpha_{i,0}=1/2, \alpha_{i,n+1}= -1/2$.   
Now set 
 \begin{alignat*}{2}\label{parstruct}&\hat\alpha_{i,0}= \alpha_{i,0} =1/2 , &&\hat F_{i,0} = F_{i,t_i} \\
&\hat\alpha_{i,1}=  \alpha_{i,t_i+j_i^1}, &&  \hat F_{i,1}=F_{i, t_i+j_i^1}  \\
&\ldots &&\ldots \\
 & \hat\alpha_{i,k_i}=  \alpha_{i,t_i+j_i^{k_i}},&&\hat F_{i,k_i}=F_{i, t_i+j_i^{k_i}}  \\
  &  \hat\alpha_{i,k_i+1}=\alpha_{i,n+1} =-1/2,\quad && \hat F_{i,k_i+1} = E_{p_i} \end{alignat*} 
This gives a nested sequence of subspaces of dimensions $s_i, s_i+j_i^1, s_i+j_i^2,\cdots, n$, with each successive quotient being non zero, except  for $\hat F_{i,k_i+1}/ \hat F_{i,k_i}$ if $t_i= 0$. Associated to each of them is a different weight. One has the immediate lemma: 

\begin{lem} The parabolic bundle $(E, F_i)$ with weights $\alpha_{i,j}$ is (semi-) stable if and only if the  parabolic bundle $(E,\hat F_i)$ with weights $\hat\alpha_{i,j}$ is.\end{lem}
Let $P_i$ be the parabolic subgroup fixing the standard flag $\bbc^{t_i}\subset \bbc^{t_i+j_i^1}\subset\cdots \subset \bbc^{t_i+ j_i^{k_i-1} }\subset\bbc^{t_i+ n-s_i }$.  Given a semi-stable parabolic vector bundle $E$, there is no difficulty in  constructing a pair $(E,\vec g)$ to which it corresponds. The set of pairs $(E,\vec g)$  corresponding to the same parabolic structure is an orbit of the group $P_1\times\cdots \times P_\ell$ acting on $(\bbc^n)^\ell$; this, combined 
with Proposition \ref{propp}, tells us that we can obtain the parabolic 
moduli space as a quotient of the Grassmannian-framed moduli space:

\begin{cor}
Let weights $\hat\alpha_{i,j}$ be given as above, from weights $\alpha_{i,j}$ satisfying $\sum_{i,j}\alpha_{i,j} = -\delta_0$, and let $P_1\times\cdots \times P_\ell$ be the parabolic subgroup of $Gl(n,\bbc)^\ell$. Let $V_{\hat\alpha}$ be the subvariety of elements of $\cagm$ satisfying the constraint \ref{closed} that map to  semi-stable parabolic bundles for the weights $\hat\alpha_{i,j}$. Then the moduli space ${\ca PM}_{\underline\alpha}$ of parabolic bundles with weights $\hat\alpha_{i,j}$ is the quotient of $V_{\hat\alpha}$ by  $P_1\times\cdots \times P_\ell$.  \end{cor}

\begin{proof} We note that the quotient in essence is modelled on quotient of the Grassmannian $Gr_n(\bbc^n_1\oplus \bbc^n_2)$ by  a parabolic $P$, which yields a particular flag manifold  on $\bbc^n$.  Indeed, if one fixes $t, s$, and a space $S$ of dimension $s$ in $\bbc_1$, the same construction as above gives us from the subvariety of planes $V_{t,s, S}$ of planes $\Pi$ in $Gr_n(\bbc^n_1\oplus \bbc^n_2)$ such that $\Pi\cap\bbc^n_2 = \bbc^t$, $\Pi\cap\bbc^n_1 = S$ a family of flags
$S\subset S_{j^1}\subset S_{j^2}\subset\cdots \subset S_{j^k}\subset \bbc^n_1$. The quotient is a homogeneous one, so stability is not really an issue. The same holds for the family in our moduli construction.  \end{proof}
We will return to the description of the parabolic moduli as quotients in sections three and four, giving a more complete description.

\section{Extended moduli spaces, and their Grassmannian 
version}\label{sec3}

\subsection{Extended moduli spaces}

We now turn to the description of the extended moduli spaces of Jeffrey, 
as explained in \cite{Je}, and then describe their Grassmannian 
compactifications. Let
$$
X^*\, :=\, X \setminus \{p_1\, ,\cdots\, 
p_\ell\}
$$
be the punctured Riemann surface. Parametrize disjoint 
neighborhoods of the punctures 
as semi-infinite cylinders, with complex coordinate 
$r+\sqrt{-1}\theta$, $r\,\in\, 
[0,\infty)$, $\theta\in \bbr/2\pi\bbz$. Choose points ${\widetilde p}_i$ 
given by $(r,\theta) =(1,0)$, thought of as close to their respective 
punctures. We consider the space ${ EM}_n$ of equivalence classes of
flat unitary connections on $X^*$, which are of the form 
$\sqrt{-1}\delta d\theta$ on the semi-infinite cylinders, where
$\sqrt{-1}\delta$ is some constant skew hermitian matrix. Here the 
equivalence 
is given by gauge transformations which are the identity on the 
semi-infinite cylinders. (This is not quite Jeffrey's construction, but 
suffices for our purposes.)

One can choose paths $c_i, i = 2,\cdots ,\ell$, from ${\widetilde p}_1$ 
to 
${\widetilde p}_i$, loops $a_j, b_j, j=1,\cdots ,g$, based at 
${\widetilde p}_1$, loops $d_i$ based at the $p_i$ around the punctures $p_i$, such 
that the 
fundamental group of $X^*$ is generated by $a_1, b_1,\cdots , 
a_g,b_g, d_1, 
c_2d_2c_2^{-1},\cdots , c_\ell d_\ell c_\ell^{-1}$ subject to the 
relation 
\begin{equation}(\prod_{j=1}^g[a_j, b_j]) 
d_1c_2d_2c_2^{-1}\cdot\ldots\cdot 
c_\ell d_\ell c_\ell^{-1}= 1.\end{equation}
One can integrate the connections. We note that there are implicit 
trivializations at each of the ${\widetilde p}_i$; these trivializations 
extend naturally to $r\,\in\, [0\, ,\infty), \theta \,\in\, 
(-\pi\, , \pi)$. In 
particular, the integration of the connections along each of the 
paths 
$a, b, c$ is well defined; along the paths $d_i$, the integral is simply 
$\exp(2\pi\sqrt{-1}\delta_i)$. Our space $EM_n$ is then the space
of elements $A_j\, ,B_j$, $j=1,\cdots ,g$, $C_i$, $i=2,\cdots ,\ell$, of 
$U(n)$ and 
$\sqrt{-1}\delta_i$, $i =1,\cdots ,\ell$, of $u(n)$ satisfying 
\begin{equation}(\prod_{j=1}^g[A_j, B_j]) \exp(2\pi\sqrt{-1}\delta_1)C_2 
\exp(2\pi\sqrt{-1}\delta_2)C_2^{-1}\cdot\ldots\cdot C_\ell \exp(2\pi 
\sqrt{-1}\delta_\ell) C_\ell^{-1}= 1\, .
\label{defining-constraint}\end{equation}

An element of ${EM}_n$ can be represented either as a triple 
$(\widetilde{E},\nabla, f)$ consisting of a unitary bundle 
$\widetilde E$, a unitary flat connection $\nabla$ of the form 
$\sqrt{-1}\delta_i d\theta$ near the punctures, and a 
unitary framing $f= (f_1,\cdots, f_\ell)$ near the punctures,
alternately, 
as a tuple $(A_j,B_j, C_i, \delta_i)$ representing the 
holonomies. Under the first representation, the infinitesimal deformations of the moduli space are given by covariant constant $u(n)$-valued one-forms $\sigma$ which are locally constant near the punctures, and of the form $a_i d\theta$. One then has, for a pair $\sigma_1, \sigma_2$ of such forms, a closed skew form
$$\Omega(\sigma_1,\sigma_2) \,=\, 
-\int_{X^*}{\rm tr}(\sigma_1\wedge 
\sigma_2)\, .$$

Jeffrey shows that the variety ${ EM}_n$ is smooth, and that the 
form $\Omega$ is non-degenerate, for the $\delta_i$ in a 
neighborhood of the origin, indeed in the neighborhood of any 
central element. For smoothness, one can isolate any one of the 
terms $\exp(2\pi \sqrt{-1}\delta_i)$ in the 
defining equation \eqref{defining-constraint}, and so the variety 
has the form of a graph and is smooth, as long as one is at a 
point at which the exponential map is locally bijective. This 
holds for the $\delta_i$ 
whose eigenvalues lie in $(-1/2, 1/2)$. Another locus at which the 
variety is smooth is that of irreducible representations. On the other hand, the 
form can degenerate when the stabilizer of one of the $\exp(2\pi 
\sqrt{-1}\delta_i)$ differs from (and so is larger than) that of 
$\delta_i$. Indeed let $\gh_{1,i}$ be the stabilizer in $\gu(n)$ 
of $\exp(2\pi \sqrt{-1}\delta_i)$, and $\gh_{2,i}$ be the 
stabilizer of $\sqrt{-1}\delta_i$;   Let $\gs_i= \gh_{1,i}\cap 
\gh_{2,i}^\perp $; then the null space for $\Omega$ is tangent to 
the 
distribution spanned by the action of $\oplus_i\gs_i$. To see this, 
we recall from Jeffrey (\cite{Je}) the various deformation 
spaces in play, and the diagram that they fit into, given in 
(3.6) of the proof of (3.1) in \cite{Je}. The main space is  
the tangent space of deformations of the relevant flat 
connections with correct asymptotic form near the punctures, 
modulo compactly supported gauge transformations; this gets 
expressed as a cohomology group $H^{1,\gg}(X^*)\,=:\,
H^{1,\gg}$. Inside this space 
there is a space $H^1_c(X^*)\,=:\, H^1_c$ of compactly supported 
deformations; the 
quotient $H^{1,\gg}/H^1_c$ maps to $\gu(n)^{\ell}$, by taking 
values 
at the boundary.  On the union $S$ of the boundary circles of 
$X^*$, one has the space $H^0(S)$ of covariant constant sections 
of the adjoint bundle, as well as the dual space  $H^1(S)$; the 
space $H^0(S)$ contains the space $H^0(X^*)$ of covariant 
constant sections over the whole punctured curve. One also has 
the space of  deformations $H^1$ of all flat connections, modulo 
all gauge transformations, on $X^*$. One has the diagram (3.6) of 
\cite{Je}, fitting all of these spaces together:
\begin{equation}
\begin{matrix}
&&0&\longrightarrow & H^1_c & \longrightarrow & H^{1,\gg} 
&\stackrel{b}{\longrightarrow}& \gu(n)^{\ell}\\
&&\Big\downarrow&& \Vert && ~\Big\downarrow \tau && 
~\Big\downarrow \sigma\\
H^0(X^*)&\stackrel{\beta}{\longrightarrow} 
&H^0(S)&\stackrel{\gamma}{\longrightarrow}& 
H^1_c&\longrightarrow &H^1 & \stackrel{\gamma^*}{\longrightarrow}
&H^1(S)&\stackrel{\beta^*}{\longrightarrow}~H^2_c
\end{matrix}
\end{equation}
Note that the arrows on the bottom are indeed duals, using 
Poincar\'e duality. Let $\psi$ be a smooth function that is one 
on the ends of the curve, and is zero on its interior; one can 
find an inverse to the map $b$ over $\bigoplus_i \gh_{2,i}^\perp$  by associating to an element $h$ in $\gh_{2,i}^\perp$ the 
deformation $f(h) = d_A( \psi (ad(\delta_i)^{-1} h))$, where 
$d_A$ is the covariant derivative of the connection. These
elements are supported over the ends; they are coboundaries in 
$H^1$, and so are mapped to zero by $\tau$. On the other hand, 
the elements of $\bigoplus_i \gh_{2,i}$ map to non-zero elements 
of $H^1(S)$. Finally, we note that the space $H^0(S)$ is spanned 
by elements $Ad(\exp(\theta\delta_i)) (s),\ s\in \gh_{1,i}$; the 
elements corresponding to $s\in  \gh_{2,i}$ are constant. These 
elements map to $-d\psi Ad(\exp(\theta\delta_i)) (s)$ in $H^1_c$.

Now consider the form $\Omega(a, \cdot)$ on $H^{1,\gg} $; if 
$\tau(a)$ is non-zero, then the form is non-degenerate, as 
$H^1_c$ is Poincar\'e dual to $H^1$. The kernel of $\tau$ is 
spanned, in turn,  by $f(\bigoplus_i \gh_{2,i}^\perp)$ and by 
$\gamma (H^0(S))$. From the explicit form of elements in 
$f(\bigoplus_i 
\gh_{2,i}^\perp)$, one can check that the form restricted to this 
subspace is non-degenerate. There remains the elements of 
$H^0(S)$. Consider those corresponding to elements of 
$\gh_{2,i}$; they are constants $s$, and map to $-sd\psi$ in 
$H^1_c$ under $\gamma$. On the other hand, the map $\sigma$ 
gives us from elements $s$ in $\gh_{2,i}$ elements $s d\theta$ in 
$H^1(S)$; if 
these elements come from elements of $H^1$, one then has a non 
zero pairing, as ${\rm tr}(s^2) $ is non-zero. For this to be the 
case, $\beta^*(s d\theta)$ must vanish. Pairing with elements 
$\alpha$ of $(H^2_c)^* \,=\, H^0(X^*)$, this tells us that $s$ 
should be orthogonal to the image of $H^0(X^*)$ in $H^0(S)$. This 
tells us that the pairing is non degenerate on the image
$\gamma(\gh_{2,i})\subset H^1_c$. 

There remains the subspaces  of $H^0(S)$ corresponding to $\gs_i$; and indeed the form can degenerate on these. 

Jeffrey also shows that the parabolic moduli spaces, in their symplectic 
description, can be obtained as symplectic quotients of ${EM}_n$, for 
weights in the open interval $(-1/2, 1/2)$ (our Grassmannian moduli 
space will allow us to extend this to the closed interval). For the 
moment, we note that associated to each puncture $p_i$, there is a 
natural action of $U(n)$ on the trivialization at ${\widetilde p}_i$. In 
terms 
of the parametrization above, if $(g_1,\cdots ,g_\ell)\in U(n)^\ell$, 
this action is given by 
$$(A_j,B_j, C_i, \delta_i)\,:=\, (\{A_j\},\{B_j\}, \{C_i\}, 
\{\delta_i\})\,\longmapsto\, 
(g_1A_jg_1^{-1},g_1B_jg_1^{-1}, 
g_1C_ig_i^{-1},g_i\delta_ig_i^{-1})\, .$$
The action is Hamiltonian, with moment map $$\nu_{U(n)^\ell}(A_j,B_j, 
C_i, \delta_i) \,= \, \sqrt{-1}(\delta_1, \cdots,\delta_\ell)\, .
$$
The parabolic moduli 
for $\delta_i$ such that ${\rm Stab}(\exp(2\pi \sqrt{-1}\delta_i)) \,=\, 
{\rm Stab}(\sqrt{-1}\delta_i)$ is then the symplectic quotient 
$\nu_{U(n)^\ell}^{-1}(\prod_j{\ca O}_{\sqrt{-1}\delta_j})/U(n)^\ell$.

We note that as a consequence of the defining constraint 
\eqref{defining-constraint}, we have
$$\sum_i \text{tr}(\delta_i) \,\in\, \bbz\, .$$
The actual value of the sum is minus the degree $\delta_0$ of the 
eventual holomorphic bundle that we will build. The integer values of 
$\delta_0$ split the moduli space into components $EM_{n,\delta_0}$. 

\subsection{Grassmannian extended moduli} 

We have for the action of $U(n)^\ell$ on $EM_n$ at the punctures the moment map:
$$\nu_{U(n)^\ell}(A_j,B_j, C_i, \delta_i) \,=\,
\sqrt{-1}(\delta_1,\cdots,\delta_\ell)$$
(see \cite{Je}).

Now consider the Grassmannian ${\rm Gr}_m(m+n)$ of $m$-planes in 
$\bbc^{m+n}$. The K\"ahler form on this 
manifold can be given either by restricting the canonical Fubini-Study 
form on projective space, or as the Kostant-Kirillov form on the 
Grassmannian thought of as the coadjoint orbit of $\lambda = 
\frac{\sqrt{-1}}{2} 
(-1,\cdots,-1,1,\cdots,1)$, under the action of $U(m+n)$. The coadjoint 
orbit is then (writing $\bbc^{m+n}$ as $\bbc^m\oplus \bbc^{n}$):
$$\Biggl\{ \begin{pmatrix}a&b\\ c&d\end{pmatrix}\cdot 
\frac{\sqrt{-1}}{2}\begin{pmatrix}-\bbi&0\\0&\bbi\end{pmatrix}\cdot 
\begin{pmatrix}a^*&c^*\\ b^*&d^*\end{pmatrix}
$$
$$
= \, \frac{\sqrt{-1}}{2} 
\begin{pmatrix}-aa^*+bb^*&-ac^*+bd^*\\-ca^*+db^*&-cc^*+dd^*\end{pmatrix}\Biggr| \begin{pmatrix}a&b\\ c&d\end{pmatrix}\in U(m+n)\Biggr\}$$
The $U(m+n)$ moment map for this is simply the identity map, i.e., the 
inclusion, and so the moment map for the action 
of $U(m)\times\{\bbi\}$ is $\frac{\sqrt{-1}}{2}(-aa^*+bb^*)$, and the moment map 
for the action of $\bbi\times U(n)$ is 
$\frac{\sqrt{-1}}{2}(-cc^*+dd^*)$. 

\begin{lem}\label{Grassmann-moment} 1) Representing the generic element 
of the Grassmannian ${\rm Gr}_m(m+n)$ of $m$-planes in $\bbc^{m+n}$ as 
the 
graph of a linear transformation $\gamma:\bbc^m\longrightarrow \bbc^n$, 
the moment map for $U(m)\times\{\bbi\}$ acting on ${\rm Gr}_n(m+n)$ 
(the ``right action") in this parametrization is then 
$$\mu_1(\gamma) = -\frac{\sqrt{-1}}{2}aa^* + \frac{\sqrt{-1}}{2}bb^*
\,=\,\frac{\sqrt{-1}}{2}(-\bbi 
+ \gamma^*\gamma)(\bbi+\gamma^*\gamma)^{-1}\, ;$$ that for $\bbi\times 
U(n)$ (the ``left action")
 is $$\mu_2(\gamma) \,=\, -\frac{\sqrt{-1}}{2}cc^* 
+\frac{\sqrt{-1}}{2}dd^*\,=\, 
\frac{\sqrt{-1}}{2}(\bbi-\gamma\gamma^*)(\bbi+\gamma \gamma^*)^{-1}.$$
 2) Representing an element of the Grassmannian as the elements annihilated by the $n$ orthonormal rows of a matrix $(b^*, d^*)$, the moment map for the action of $U(m)\times\{\bbi\}$ is
 $$\mu_1(b^*, d^*) \,=\, \frac{\sqrt{-1}}{2}(-\bbi+2 bb^*)\, , $$
 and that of $\bbi\times U(n)$ is
 $$\mu_2(b^*, d^*) \,=\, \frac{\sqrt{-1}}{2}(-\bbi+2 dd^*)\, . $$
 \end{lem}

\begin{proof}
If we represent the generic element of the Grassmannian 
${\rm Gr}_m(m+n)$ as the graph of $\gamma:\bbc^m\longrightarrow \bbc^n$, 
the 
$m$-plane is spanned by the columns of 
$$ \begin{pmatrix}\bbi\\ \gamma\end{pmatrix}$$ or, equivalently, by the 
(mutually orthogonal) columns of 
$$ \begin{pmatrix}\bbi\\ 
\gamma\end{pmatrix}(\bbi+\gamma^*\gamma)^{-1/2}\,=\, \begin{pmatrix}a\\ 
c\end{pmatrix}\, .$$ 
The orthogonal complement of this is spanned by 
$$ \begin{pmatrix}-\gamma^*\\ \bbi 
\end{pmatrix}(\bbi+\gamma\gamma^*)^{-1/2}\,=\, \begin{pmatrix}b\\ 
d\end{pmatrix}\, .$$ 
We then apply the formulae. In the same vein, the correspondence 
between the coadjoint orbit and the Grassmannian is by taking our 
$m$-plane to be the $-\sqrt{-1}/2$ eigenspace of the matrices in the 
orbit under the action of $U(n+m)$; parametrizing as above, the 
eigenspace is annihilated by the rows of $(b^*, d^*)$. On the other hand, as the matrix is unitary, we have $aa^* + bb^* = \bbi, cc^*+dd^* = \bbi$, and so 
\begin{align} \mu_1(b^*, d^*) =& \frac{\sqrt{-1}}{2}(-aa^*+bb^*) = 
\frac{\sqrt{-1}}{2}(-\bbi+2 bb^*),\nonumber \\
 \mu_2(b^*, d^*) =& \frac{\sqrt{-1}}{2}(-cc^*+dd^*) = 
\frac{\sqrt{-1}}{2}(-\bbi+2 dd^*)\, .\nonumber\end{align}
We note that the image of both moment maps is the set of skew hermitian matrices with eigenvalues in the interval 
$\sqrt{-1}\cdot [-1/2,1/2]$ (We note that we can scale the symplectic form, by changing the coadjoint orbit; the scale here is chosen so that the image of the moment map is the $U(n)$ orbit of the fundamental alcove for $U(n)$.) 

We now want to replace framings by their graphs. We do this by acting
 by $U(n)^\ell$ diagonally on $EM_{n,\delta_0}\times 
{\rm Gr}_n(2n)^\ell$, 
with the right action (by $(U(n)\times \{\bbi\})^\ell$) on the Grassmannians, and reducing at zero. This means that one considers elements
$(A_i,B_i, C_i, \delta_i)(b^*_i, d^*_i)$ lying in the zero locus $M^{-1}(0)$ of the moment map, that is, satisfying 
\begin{equation}\delta_i \,= \,\frac{ {1}}{2}(\bbi-2 b_ib_i^*)\,
,\label{mo}\end{equation} and quotients by the action of $U(n)^\ell$.
\end{proof}

We note that the diagonal $S^1$ of $(U(n)\times \{\bbi\})^\ell$ acts trivially on $EM_{n,\delta_0}$, but not on ${\rm Gr}_n(2n)^\ell$; also, similarly to the holomorphic case, the quotient by $(U(n)\times \{\bbi\})^\ell$ can also be thought of as a quotient by $(SU(n)\times \{\bbi\})^\ell\times S^1)$, where $S^1$ is the diagonal $S^1$ of $(\bbi\times U(n))^\ell$.

\begin{prop}
The symplectic quotient at zero
$$ GM_{n,\delta_0} = (EM_{n,\delta_0}\times 
{\rm Gr}_n(2n)^\ell)/\!\! /U(n)^\ell$$
is smooth when
 
1) at least one of the $\delta_i$ has all its eigenvalues in $(-1/2, 1/2)$, or

2) when the representation is irreducible, and at least one eigenvalue 
of at least one $\delta_i$ is not $\pm 1/2$.

Over the locus where the moment map is submersive, the form is symplectic over the quotient. 

The right action of the group $U(n)^\ell$ descends to the 
quotient $GM_{n,\delta_0}$, with moment map
$$(A_j,B_j, C_i, \delta_i)(b^*_i, d^*_i)\,\longmapsto 
\,\frac{\sqrt{-1}}{2}(\bbi-2d_id_i^*) \label{moment-GM}\, .$$
\end{prop}
 
\begin{proof}
We first note that $EM_{n,\delta_0}$ is smooth when either 
condition 1) or 2) is satisfied. The stabilizer in $U(n)^\ell$ 
of an element in $EM_{n,\delta_0}$ is a diagonal embedding of 
the automorphism group of the representation; on the other hand, 
conditions 1 or 2) then guarantee that the action of this 
automorphism group on the 
corresponding elements in the Grassmannians is free.
The smoothness statement follows. If the moment map on the 
product is submersive at a 
point, one 
then has in the usual way that the degeneracy locus for the symplectic form restricted to $M^{-1}(0)$ is precisely the $U(n)^\ell$ orbit, and so the form on the quotient is indeed non degenerate, even if it is not on one of the factors. The action of $U(n)^\ell$ on the quotient simply follows from the two commuting $U(n)$ actions on the Grassmannian.
\end{proof}

 Let $s_i$ denote the dimension of the $1/2$-eigenspace of 
$\delta_i$, and $t_i$ the dimension of the $-1/2$-eigenspace. The 
relation \eqref{mo} tells us that $s_i$ is the dimension of the kernel of $b^*_i$, and $t_i$ is the dimension of the subspace of $\bbc^n$ of vectors whose norm is preserved by $b^*_i$; this dimension is then the dimension of the kernel of $d^*_i$, as the rows of $(b_i^*, d^*_i)$ are orthonormal. Thus, $s_i$ is the dimension of the intersection of the plane with the first copy of $\bbc^n$ in $\bbc^{2n}$, and $t_i$ is the intersection of the plane with the second copy. We have
$$\sum_{i=1}^{\ell} {\rm tr}(\bbi/2 + \delta_i)\, = \, \frac{n\ell}{2} 
-\delta_0\, .$$
This value implies the constraints
\begin{align}\label{stconstraints}
 \sum_{i=1}^{\ell}s_i\quad (\leq) \quad & \frac{\ell n}{2} - 
\delta_0\, , \\
 \sum_{i=1}^{\ell}t_i\quad (\leq) \quad & \frac{\ell n}{2} + 
\delta_0\, ,\nonumber\end{align}
 as in \eqref{1-s-bound}, \eqref{1-t-bound}. We note that the 
relation 
\eqref{mo} relates the trace of $\delta_i$ to that of $b_ib^*_i$ and so of $b^*_ib_i$. The relation 
$b_i^*b_i + d_i^*d_i = \bbi$ (orthonormality of the rows of $(b_i^*, d^*_i)$) then ties this to the trace of $d_i^*d_i$ and so of $d_id_i^*$.

\subsection{Reduction by $U(n)$: parabolic moduli spaces}

We next consider the reductions of $ GM_{n,\delta_0}$  by the remaining action of $U(n)^\ell$; this is the "left" action, by $(\{\bbi\}\times U(n))^\ell$. This reduction will  give us the parabolic moduli spaces. For this, we first give a lemma.
\begin{lem} \label{normalform}Let a plane in the Grassmannian $Gr_m(m+n)$ of $m$-planes in $\bbc^{m+n}$ be the kernel of a matrix $\rho^* = (b^*, d^*)$  with $n$ orthonormal rows, with $b^*$ an $n\times m$ matrix, and $d^*$ an $n\times n$ matrix, satisfying $\rho^* \rho= \bbi$.
Then, acting by $U(m)\times U(n)$ on the right, and $U(n)$ on the left, we can normalize $\rho^*$ to the form:
$$\rho^*= \begin{pmatrix} 0&0&0&0&\bbi &0&0\\0&0&(\frac{\bbi}{2}-{ 
\gamma} )^{1/2}&0&0&(\frac{\bbi}{2}+{  
\gamma} )^{1/2}&0\\0&0&0&\bbi&0&0&0\end{pmatrix}\, .$$
Horizontally, one has blocks of size $m-n,s,r, t, s, r, 
t$ and vertically, of size $s,r, t$, with $s+r+t = n$. The matrix $\gamma$ is diagonal, with entries in $(-1/2, 1/2)$.
\end{lem}
\begin{proof}We first use the right $U(m)$ action to normalize the kernel $K$ of $b^*$, which we suppose of dimension $m+s$, to the span of the first $m-n+s$ vectors of the basis  and  the right $U(n)$ action to normalize the kernel of $d^*$ to the span of the last $t$
  of the basis of $\bbc^n$; as the columns of $b^*, d^*$ span $\bbc^n$, we can take the intersection of the images to be the middle $r$ vectors of the basis, and the span of the columns of $d^*$ to be  the first $r+s$ vectors. Using the relation $\rho^* \rho= \bbi$, one has that the last $t$ rows of $b^*$ are then orthonormal, and so one can use the right $U(m)$ action so that the image of these last $t$ rows are the last $t$ vectors of a basis of $\bbc^m$; in a similar way, the first $s$ rows of $d^*$ are orthonormal, and one can use the right $U(n)$ action to move the image of these  to the first $s$ vectors. This then gives a new $\rho^*$ of the form:
$$\rho^*\,=\, \begin{pmatrix} 0&0&0&0&\bbi  &0&0\\ 0&0&M&0&0 &N  &0\\ 
0&0&0&\bbi&0&0&0\end{pmatrix}\, .$$

We can further normalize $M$: first use the remaining freedom in the  $U(m )$ action (on the 
right on $M$) to make $M^*M$ a positive, diagonal matrix; this then 
means that $MD$ is unitary for a positive diagonal matrix $D$; now use 
the left $U(n)$ action to map $MD$ to the identity, and so $M$ to a positive diagonal matrix . Then,  however, $M^2 + NN^* = \bbi$, so that $NN^*$ is diagonal, positive. Using the remaining right $U(n)$ action, one can make $N$ positive, diagonal, and of course with $M^2+N^2= \bbi$. \end{proof}

Now let us consider the reductions of $ GM_{n,\delta_0}$ by $(\{\bbi\}\times U(n))^\ell$. This involves fixing the image of the moment map $(A_j,B_j, C_i, \delta_i)(b^*_i, d^*_i)\,\longmapsto 
\,\frac{\sqrt{-1}}{2}(\bbi-2d_id_i^*) $, and quotienting by the stabilizer of the image. We can take the $\delta_i$ to be diagonal, and the reduction used to build $ GM_{n,\delta_0}$, which sets $\delta_i = (\bbi-2b_ib_i^*)$ tells us that the eigenvalues of $\delta_i$ lie in $[-1/2,1/2]$. As shown in the preceding lemma, we can normalize the  $b_i$, $d_i$ so that they are diagonal, positive; the matrices $\delta_i, b_i^*, d_i^*$ are related by $b_i= (\frac{\bbi}{2}-\delta_i)^{\frac{1}{2}}, d_i= (\frac{\bbi}{2}+\delta_i)^{\frac{1}{2}}$, and so the moment map $\frac{\sqrt{-1}}{2}(\bbi-2d_id_i^*) $ takes value $\sqrt{-1}\delta_i$. In essence, the lemma tells us that the Grassmannian reduction contains no extra information; the reduction of $ GM_{n,\delta_0}$ at $\sqrt{-1}\delta_i$ is the same space as that of $ EM_{n,\delta_0}$. In short, Jeffrey's result for $ EM_{n,\delta_0}$ yields:
\begin{prop} The reduction of $ GM_{n,\delta_0}$ by the action of $U(n)^\ell$ at $\sqrt{-1}\delta_i$ is the parabolic moduli space $PM_{\underline\delta}$ corresponding to $\delta_i$, that is with weights given by the eigenvalues of $\delta_i$\end{prop}

We note that the eigenvalues of the $\delta_i$ are allowed to take values of $1/2$, $-1/2$; this is contrary to the usual practice of the parabolic moduli, for which the eigenvalues all lie within an open unit interval. From the complex point of view, the elements of the moduli    when both $1/2$, $-1/2$ occur will  involve a  (partial) Hecke transform of the elements of the  more standard moduli, with, say, only $-1/2$ occurring.

\section{The correspondence}\label{sec4}

We have defined a symplectic moduli space ${GM}_{n,\delta_0}$, as well as a holomorphic one $\cagm$. We now want to define a map $C$ between the two, and show that it is a homeomorphism.  As reductions on both sides correspond to parabolic moduli, we will want this map to commute with these reductions.

Let $({\widetilde E},\nabla\, , {\widetilde f}\, , {\widetilde g})$ in  $EM_{n,\delta_0}\times Gr_n(2n)^\ell$ represent an element of ${GM}_{n,\delta_0}$, and consist  of a bundle with flat unitary connection $({\widetilde 
E},\nabla)$ on the punctured curve,  unitary trivialisations ${\widetilde f} =({\widetilde f}_1,...,{\widetilde f}_\ell), {\widetilde f}_i:{\widetilde E}|_{{\widetilde p}_i}\rightarrow \bbc^n$ at ${\widetilde p}_i$, and elements ${\widetilde  g}= ({\widetilde  g}_1,..,{\widetilde  g}_\ell)$ of the Grassmannian ${\rm Gr}_n(\bbc^n\oplus\bbc^n)$. There is a natural map, associating to these a
pair $$C(X) \,=\, (E, \vec g)$$ consisting of a bundle
$E$ on $X$ and Grassmannian framings $ g_i$ at $p_i$ of $E$.  

To do this,
we first define the bundle $E$.  Note that the flat unitary bundle $\widetilde E$ gives a holomorphic bundle 
 in a natural way on the punctured curve $X^* = 
X\setminus\{p_1\, ,\cdots\, ,p_\ell\}$; the connection gives a $\bar\partial$-operator; in particular, covariant constant sections are holomorphic. Unfortunately, a covariant constant section has monodromy $exp(2\pi\sqrt{-1}\delta_i)$ around the $i$-th puncture, and does not extend over the puncture. If $z$ is a holomorphic coordinate with $z=0$ corresponding to the puncture,   the   monodromy is that of the function $z^{\delta_i} = exp (ln(z)\cdot \delta_i)$, and so  one can  extend $\widetilde E$  to a 
holomorphic bundle on $X$ by glueing 
$\widetilde E$ to the trivial bundle over a  disk centered on the puncture via a transition 
function $z^{\delta_i} = exp (ln(z)\cdot \delta_i)$, so that if $e_i$ is a basis of the trivial bundle, and $f_i$ is the covariant constant (unitary) basis of $\widetilde E$, one identifies $e_i$ with $z^{-\delta_i}\cdot f_i$. 

We note that this procedure gives us a flag at the puncture in the fiber of $E$, as well, following the procedure used in the Mehta-Seshadri construction of parabolic bundles. We represent as in the previous section $({\widetilde E},\nabla\, , {\widetilde f})$ as a tuple of holonomies $(A_j,B_j, C_i, \delta_i)$, and we note that there are basically two flags of subspaces at the puncture ${\widetilde p}_i$: the standard flag $F_s$ and the flag $F_\delta$ of eigenspaces of $\delta_i$, ordered by increasing eigenvalues; it is the latter that we want. The symplectic reduction used in \cite{Je} works by fixing the eigenvalues of $\delta_i$, then restricting to the subvariety for which
 $\delta_i$ is of the form  $\diag(\delta_{i,1},\cdots ,\delta_{i,n})$, with $\frac{1}{2}\geq \delta_{i,1}\geq\delta_{i,2}\geq \cdots \geq\delta_{i, n}\geq\frac{-1}{2}$ (incidentally setting $F_\delta$ to be a subflag of $F_s$) , and finally quotienting by the stabilizer of these $\delta_i$. 
 
 The construction of the flag as given in section 2, on the other hand, when adapted to the symplectic context, gives us in essence the flag $F_s$. To see this, we represent the elements ${\widetilde  g}_i$ of the Grassmannian by matrices $(b^*_i, d^*_i)$ whose kernel is the element of the flag, satisfying $b^*_ib_i+ d^*_id_i = \bbi$; let us assume, to illustrate what is happening,  that $b^*_i, d^*_i$ are invertible; the two moment maps for the left and right actions of the Grassmannians are related by   
\begin{align}\mu_1(b_i^*, d_i^*) \,=&\, \frac{\sqrt{-1}}{2}(-\bbi+2 b_ib_i^*)\,\nonumber \\
 =&\, - (b_i^*)^{-1}d_i^*(\mu_2(b_i^*, d_i^*))(d_i^*)^{-1}b_i^* \,=\, - (b_i^*)^{-1}d_i^*(\frac{\sqrt{-1}}{2}(-\bbi+2 d_id_i^*))(d_i^*)^{-1}b_i^*\, .\nonumber \end{align}
Recall that we were quotienting by $\mu_1$ to define $ GM_{n,\delta_0}$, setting  $\mu_1(b_i^*, d_i^*)= -\delta_i$.  The above tells us that the $U(n)$ moment map on $ GM_{n,\delta_0}$, which is $\mu_2$, is a conjugate of $\delta_i$. The map $(b_i^*)^{-1}d_i^*$ in essence maps the standard trivialisation to the trivialisation $f_i$, and takes the standard flag in $\bbc^n$ to the flag $F_s$. This shows that the flags defined in section 2 are in essence subflags of $F_s$; for these to be $F_\delta$, we will restrict, in our construction of $C$, to the analytic subset
$$EG_\diag = \{(A_j,B_j, C_i, \delta_i)(b_i^*, d_i^*)| \delta_i = \diag(\delta_{i,1},\cdots ,\delta_{i,n}),\frac{1}{2}\geq \delta_{i,1}\geq\delta_{i,2}\geq \cdots \geq\delta_{i, n}\geq\frac{-1}{2}\}$$
As we want to define $C$ on the quotient $ GM_{n,\delta_0}$, this poses no problem as along as our construction is equivariant under the   unitary stabilizers of the $\delta_i$.

We would thus like to transfer the trivialization, and more generally, the Grassmann framing, to the bundle $E$ over the  puncture, though one must ``renormalize" by the decay. Indeed, an 
element $({\widetilde E},\nabla\, , {\widetilde  f}\, , {\widetilde  g})$ of $EM_{n,\delta_0}\times Gr_n(2n)$ defines unitary covariant 
constant trivializations $$f_i: E\longrightarrow \bbc^n$$ extending the ones over the points $\tilde p_i$ to
the set 
$r= -ln(|z|) \in [0,\infty), \theta = arg(z) =0$ near the $i$-th puncture; as we have seen, this does not extend over $r=\infty$ . We re-scale this  to a trivialization $e_i$ over the real half line which does extend to the puncture $p_i$, by composing $f_i$ with $z^{-\delta_i}$, as we did above; the rescaling ensures that we have a well defined trivialization in the limit. This construction, by the way, on the level of reduced subspaces, reproduces the classical construction of Mehta and Seshadri, and allows definition of a natural flag over the puncture, with the flag of eigenspaces of $\delta_i$  corresponding to flag over the punctured defined by the decay rates of the  flat sections $f_{i,j}$.

We can also transfer our element ${\widetilde  g}_i$ in the 
Grassmannian ${\rm Gr}_n(\bbc^n\oplus\bbc^n)$ to an element $ 
g'_i$ in the 
Grassmannian ${\rm Gr}_n(E_{p_i} \oplus\bbc^n)$; if ${\widetilde  g}_i$ is the graph of a map 
$${\widetilde \gamma}_i: 
\bbc^n\longrightarrow\bbc^n\, ,$$ $g'_i$ 
is the graph of $\gamma_i = {\widetilde\gamma}_i\cdot 
e_i$.This transfer is  equivariant under the stabilizer of the $\delta_i$, and so descends to a map at the level of ${GM}_{n,\delta_0}$. Under this, the flag given by our construction in section 2 corresponds to the Mehta-Seshadri flag.

We note in passing
\begin{prop}
The map $C$ is continuous.
\end{prop}
\begin{proof} The construction of $C$ mirrors that of many variants of $C$, most notably that given in Mehta and Sexhadri in \cite{MS}, theorem 5.3 for the case of parabolic bundles. Families in $EM_{n,\delta_0}\times Gr_n(2n)^\ell$ have a well defined local lift into the universal family parametrising the GIT quotient, as for the maps $\pi,\pi_0$ in  \cite{MS}. 

One can, in defining the map, be quite explicit, in terms of transition functions.  One has the usual construction of the punctured $X^*$ from a planar $4g+ 2\ell$-gon by glueing the sides in an appropriate way. Let us enlarge the polygon a bit to $\tilde X^*$, so that $X^*$ is constructed by glueing open strips $U_\alpha$, $U_{\alpha'}$ of $\tilde X^*\rightarrow X^*$, where the strips are thickening of the edges. Let $({\widetilde E}(t),\nabla(t)\, , {\widetilde f}(t)\, , {\widetilde g}(t))$ be a family in $EM_{n,\delta_0}\times Gr_n(2n)^\ell$. The bundles $E(t)$ over $X^*$ are constructed from the trivial bundle on $\tilde X^*$ by using  as transition functions from $U_\alpha$  to $U_{\alpha'}$ the integrals $M_\gamma(t)$ of the connections $\nabla(t)$ along loops $\gamma(\alpha)$. These vary continuously in $t$. Similarly, in extending to $X$, one glues the bundle $E$ to the trivial bundle over the disks around the punctures with transition functions $z^{\delta_i(t)}$, which in fact are defined over   $\tilde X^*$. These again vary continuously  in $t$. The map on the Grassmannian factors are just the identity when viewed in the trivializations over the disk. In short, everything is explicitly continuous.\end{proof}

We next discuss stability, semistability, and polystability, in partiicular the notion of polystability. The notion of polystability   requires the notion of sum, and this exists for our   framed bundles: indeed, if $Gr_k(V)$ denotes the Grassmannian of $k$-planes in a vector space $V$, one has a natural inclusion $I: Gr_k(V)\times Gr_{k'}(V')\rightarrow Gr_{k+k'}(V\oplus V')$. This extends to Grassmann framed bundles: if $(E,\vec{g}), (E', \vec{g'})$ are two Grassmann framed bundles of ranks $n,n'$, one defines their sum as the pair consisting of the sum of the bundle $E\oplus E'$ and the sum of the framings $g_i\oplus g'_i\in Gr (E_{p_i}\oplus E'_{p_i}\oplus \bbc^{n+n'}$.

 For parabolic bundles,   the map $C$, or more precisely its reduced version, gives polystable bundles of parabolic degree zero. Indeed, there is a parabolic degree for bundles, combining the ordinary degree and the weights. It can be computed as an integral of the curvature of a suitable singular connection, compatible with the parabolic structure. Here, as the connection is flat, the parabolic degree is zero. One also has a parabolic slope, given by the quotient of parabolic degree by rank. Parabolic irreducible connections give stable bundles, and decomposable ones give sums of stable bundles with the same parabolic slope, which here is zero; these are known as polystable.

On the other hand, going from a parabolic structure to a Grassmannian structure, tends to stabilize the structure as we have seen in proposition (\ref{propp}). For example, for $s_i=t_i = 0$, the sum of two stable ordinary bundles with same slope, which would just be semi-stable as a bundle, when lifted to the Grassmannian moduli,  gives a stable Grassmannian structure. Also,  polystable parabolic bundles which are not stable have more automorphisms than stable ones; on the other hand, lifting from parabolic structures to Grassmannian framings,  if at only one of the punctures one has $s_i=t_i = 0$, any sum has no automorphisms.

For parabolic bundles, a polystable, but non stable bundle is semistable, and each summand destabilizes the whole. Let us adopt the same definition for Grassmannian framed bundles.  
 
\begin{definition} A  Grassmannian framed bundle is { \it polystable} if it is stable or if it is semistable and a sum of Grassmann framed subbundles, each of which destabilizes the whole. \end{definition}

These bundles are quite rare;  on the other hand,   one can have a sum of Grassmann framed bundles which are (semi)-stable; we will call these {\it (semi)-stable sums}. 

To get an idea of what polystability implies in the Grassmannian framed context, we have:
\begin{prop}
Let $l\geq 2$. Let $(E,{\vec g})$, of rank $n$, degree $\delta_0$, be a semistable sum of $(E',{\vec g}')$ and $(E'',{\vec g}'')$ of ranks $n', n''$ and degrees $\delta'_0, \delta''_0$ respectively. If both $(E',{\vec g}')$ and $(E'',{\vec g}'')$ destabilize $(E,{\vec g})$, the ordinary slopes satisfy $\frac{\delta_0}{n} = \frac{\delta'_0}{n'} =\frac{\delta''_0}{n''}$. Furthermore, $\sum_i (s'_i +t'_i) = n'\ell, \sum_i (s''_i +t''_i) = n''\ell$, so that the planes are in some sense maximally degenerate.
\end{prop}
As the subbundles are destabilizing, one has the  equalities 
\begin{align}
0\, =\,& \frac{\delta_0}{n} -  \frac{\delta'_0}{n'} + \sum_i\Bigl[(\frac{n' -s'_i}{n'}) - ( 
\frac{n'-s'_i+t''_i}{n}) \Bigr]\, ,\nonumber \\
0\, =\,& \frac{\delta_0}{n} -  \frac{\delta"_0}{n''} + \sum_i\Bigl[(\frac{n'' -s''_i}{n''}) - ( 
\frac{n''-s''_i+t'_i}{n}) \Bigr]\, ,\nonumber\\
0\,=\,& \Bigl[\frac{\delta_0}{n} - \frac{\delta'_0}{n'} 
\Bigr]\Bigl[-\frac{\delta_0}{n} - \frac{1}{2} + \sum_i( \frac{n'-s'_i +t''_i}{n})\Bigr]\, ,\nonumber\\
0\,=\,& \Bigl[\frac{\delta_0}{n} - \frac{\delta''_0}{n''} 
\Bigr]\Bigl[-\frac{\delta_0}{n} - \frac{1}{2} + \sum_i( \frac{n''-s''_i +t'_i}{n})\Bigr]\, ,\nonumber
\end{align}
Let us suppose, from the last two equations above,  that 
$$\Bigl[-\frac{\delta_0}{n} - \frac{1}{2} + \sum_i( \frac{n'-s'_i +t''_i}{n})\Bigr]= \Bigl[-\frac{\delta_0}{n} - \frac{1}{2} + \sum_i( \frac{n''-s'_i +t''_i}{n})\Bigr] = 0$$
Substituting in the first two equations from these two equations gives $- \delta_0 +n\ell - \sum_is_i -n/2 = 0$, and since $\sum_is_i\leq \ell n/2 - \delta_0$, one has $n(\ell-1) \leq 0$, a contradiction if $\ell \geq 2$. This forces one of $\frac{\delta'_0}{n'}, \frac{\delta''_0}{n''}$, and hence both, to be equal to $\frac{\delta_0}{n}$.
From the first two equations above,  multiplying the first by $nn'$, the second by $nn''$, then adding, one has
$$0 = 2\ell(n'n'') - \sum_i (s'_i+t'_i)n'' -\sum_i (s''_i+t''_i)n'$$  
The only way this can be satisfied is $\sum_i (s'_i +t'_i) = n'\ell, \sum_i (s''_i +t''_i) = n''\ell$.

Returning to the map $C$, we have:

\begin{prop}
The   pairs $(E,{\vec g'})$ obtained from the map $C$ are semistable. They  are stable if the representations on the symplectic side are irreducible and if $\sum_is_i< n\ell/2-\delta_0$, $\sum_it_i<n\ell/2+\delta_0$ . They are semistable sums if the representations on the symplectic side are  reducible.
\end{prop}

\begin{proof} The $\bbc^*$ stability constraints follow from \ref{stconstraints}.
When the weights $\delta_{i,j}$ lie in the open interval $(-1/2, 1/2)$, as noted above,
by standard results for parabolic bundles going back to the original article of Mehta and Seshadri \cite{MS}, the bundle is parabolic 
polystable, and stable if irreducible. 
It is then, as we have seen in Proposition \ref{propp}, 
Grassmann-framed semistable, and stable in case of parabolic stability. 
The   arguments of Mehta and Seshadri
can apply when the weights take values at the ends of the interval: the parabolic degree is the integral of the trace of 
the curvature over the punctured curve, and so vanishes for our flat bundles; any subbundle of the flat bundle 
$E$, on the other hand,  will have negative or zero curvature over $X^*$ and can only be zero if it is invariant under the connection, and so the 
parabolic degree of the subbundle is bounded above by zero.
\end{proof}

 We note that stability is invariant under the (left or right) action of $Gl(n,\bbc)$ on the Grassmannian, as the stability only depends on the $s_i, t_i$. 
There is one final modification we want to make. Indeed, our bundle $E$ on the holomorphic side is a quotient of a fixed space $V$; let us fix a Hermitian metric on $V$. This induces a Hermitian metric $<,>$ on the fibers of the bundles $E$ over the $p_i$. The trivialisation $e_i $ is not unitary for this metric; it does however, generate the flag $F_s$, in the sense that the subspace of dimension $a$ is generated by the first $a$ vectors of the basis. We  apply  the Gram-Schmidt process to obtain a unitary basis compatible with the flag, so that $\hat e_i = \rho\cdot e_i  $ with $\rho$ an upper triangular matrix, with real diagonal entries. If the   Grassmann element $g'_i$ in $Gr_n(E_{p_i} \oplus \bbc^n)$ is the graph of $\gamma_i = {\widetilde g}_i\cdot e_i$ in the $e$-trivialisation , we modify it to $g_i$,  the graph of  $\gamma_i = {\widetilde g}_i\cdot \rho\cdot e_i =  {\widetilde g}_i\cdot \hat e_i$. Alternately, if the element $g'_i$ is the kernel of $(b_i^*, d_i^*)$ in the $e_i$-trivialisation, we modify it to $(b_i^*\rho, d_i^*)$. The construction is equivariant under the action of the unitary stabilizer of $\delta_i$, and so descends to $GM_{n,\delta_0}$. 

We have thus built a continuous map, following the general scheme of the Narasimhan-Seshadri correspondence:
\begin{align}\label{NS-correspondence}
C\,:\, GM_{n,\delta_0}\,&\longrightarrow\, \cagm\, .\\
({\widetilde 
E},\nabla\, , {\widetilde f}\, , {\widetilde g})&\mapsto (E,\vec g)\nonumber
\end{align}
The next step, naturally, is to show that this map is a homeomorphism. 
As the spaces $GM_{n,\delta_0}$ are compact, it suffices that the map 
be bijective.

The left hand side of \eqref{NS-correspondence}, as we saw, has a natural action of $U(n)$, as well 
as a moment map, given by \eqref{moment-GM}. On the other hand, as we 
saw, the resulting element of $\cagm$ is encoded as an element $(\alpha,\beta)$ of 
$P\times Q^\ell$; this is a K\"ahler manifold, once one has fixed a 
Hermitian form on $V$, and so on $V^*\oplus \bbc^n.$ In $Q$, in 
particular, one is looking at the Pl\"ucker embedding of the 
Grassmannian of $n$ dimensional planes in $V^*\oplus \bbc^n$. The group 
$U(n)$ acts on the $\bbc^n$-factors, and so on the Grassmannians 
${\rm Gr}_n(V^*\oplus \bbc^n)$.  Once one has fixed a Hermitian form   on $V$, there is also a K\"ahler moment $\mu'_{n,\ell}$ map for the $U(n)^\ell$ action.

We note that the  map $C$ is $\{\bbi\}\times U(n)^\ell = U(n)^\ell$ equivariant, and even $ Gl(n,\bbc)^\ell$ equivariant.

\begin{prop}\label{samemoment}
Let us consider a pair $(E,\hat g)$ lying in the image of $C$, 
corresponding to an element $(\alpha,\beta) = 
(\alpha,(\beta_1,\cdots,\beta_\ell))$ of $P\times Q^\ell$. The 
moment map for the 
action of $U(n)^\ell$ on $(\alpha,\beta)$ coincides with the one 
on $GM_{n,\delta_0}$ under the map $C$.

If  the $\beta_i$ are the top exterior powers of the 
orthonormal rows of $n\times (p+n)$-matrices $(b_i^*, d_i^*)$ in 
a basis for $V\oplus \bbc^n$, the moment map is given by:
$$\mu'_{n,\ell} (\alpha, \beta) \,=\, \frac{\sqrt{-1}}{2}((-\bbi+2 
d_1d_1^*),\cdots, (-\bbi+2 d_\ell 
d_\ell^*))\, .\label{moment-cal-GM-2}$$
\end{prop}

\begin{proof}
The symplectic structure and the moment map for the Grassmannian under the 
Pl\"ucker embedding is the same as for its identification as a coadjoint 
orbit; see Huckleberry, \cite[ pp. 127--128]{Hu}. We note that the map $C$ takes   planes identified as the kernels of matrices $(b_i^*, d_i^*)$  with orthonormal rows to planes identified as the kernel of matrices $(0, b_i^*, d_i^*)$, again with orthonormal rows. The identity of the moment maps, as well as their explicit form, follows from Lemma \ref{Grassmann-moment}. \end{proof}

We thus have a commuting diagram
\begin{equation} \begin{matrix} GM_{n,\delta_0}&{\buildrel{C}\over{\rightarrow}}&\cagm\\ \\ \downarrow \mu_{n,\ell} &&\downarrow \mu'_{n,\ell} \\ \\ \gu(n)^\ell&=& \gu(n)^\ell\end{matrix}\end{equation}

We note, referring to the form of the moment map that $s_i$ is the number of $1/2$ eigenvalues of the $i-$th component of the moment map, and $t_i$ is the number of $-1/2$ eigenvalues. 

We now want to see that the map $C$ is bijective. Our strategy is as follows:

\begin{itemize}
\item{} Show that any element $\ca F$ in $( \mu'_{n,\ell})^{-1}(\delta_{i})$ is represented by an $(\alpha, \beta)$ that defines a polystable parabolic bundle $\ca E$ for the choice of weights given by the eigenvalues of $\delta_i$. This is the first proposition below.
\item{} Use the known results of Biquard, Poritz et al \cite{Bi, Po} on parabolic bundles to show that there is a flat parabolic structure $E$ corresponding to $\ca E$ in  $GM_{n,\delta_0}/\!\! /_{\delta_i} U(n)^\ell$. This then means that there is an $F$ with $C(F) = \ca F$, and so $C$ is surjective. This is the first part of the second proposition below.
\item{} For injectivity, note that if $C(F_1)= C(F_2)$, then by known results for parabolic bundles, they define the same parabolic structure in $GM_{n,\delta_0}/\!\! /_{\delta_i} U(n)^\ell$; in other words, they are the same up to the action of $Stab(\delta_i) \subset U(n)^\ell$. But then they must be the same as Grassmann framed bundles; this is the second part of the second proposition below.
\end{itemize}

We will  exploit the equivalence between algebraic and K\"ahler 
quotients, as expounded by Mumford, Guillemin and Sternberg, 
\cite[Appendix 2C]{MFK}, \cite[p. 102]{Ki}. We 
note that our previous construction of $\cagm$ can be viewed as a 
K\"ahler quotient: we have an action of $SU(p)\times S^1$ on a
K\"ahler manifold of a product of Grassmannians, with moment maps $\mu_p, \mu_{S^1}\, = 
\, \sqrt{-1}\sum_itr(\delta_i)$, and $\cagm$ can be obtained as 
$\mu_p^{-1}(0)\cap \mu_{S^1}^{-1}(-\sqrt{-1}\delta_0)/(SU(p)\times S^1)$. 
We have, as the actions commute, that the moment map $\mu_{n,\ell} $ is invariant under $SU(p)$. Therefore, writing $\cagm$ as the $SU(p)\times S^1$ quotient, the moment map $\mu'_{n,\ell} $ descends to $\cagm$:
$$\mu'_{n,\ell} :\cagm \longrightarrow u(n)^\ell\,$$
and $\cagm$ satisfies the $\mu_{S^1}$ moment constraint 
$\sum_itr(\delta_i)+\delta_0 = 0$.

{}From an algebraic point of view, to obtain parabolic moduli we are quotienting the product 
$\cagm\times \prod_i{\ca O}_{-\sqrt{-1}\delta_i}$ by an action of 
${\rm GL}(n,\bbc)^\ell$; this quotient is stratified by the dimensions 
of 
intersections $s_i, t_i$ with $E_{p_i}, \bbc^n$. We note that from a complex point of view, the coadjoint orbit ${\ca O}_{-\sqrt{-1}\delta_i}$ is a flag manifold
$$Fl_i (\bbc^n) \,=\, Fl_{(n-n_{i,0}), (n-n_{i,0}-n_{i,1}), \cdots, 
(n-n_{i,0}-\ldots -n_{i,k_i}) \,=\, t_i}(\bbc^n)\, .$$ Here the 
subscripts 
denote the codimensions of the planes. One has an embedding of the flag manifold into a product of Grassmannians, for any vector space $F$:
\begin{equation}Fl_i (F) =Fl_{(n-n_{i,0}), 
(n-n_{i,0}-n_{i,1}),\cdots, 
(n-n_{i,0}-\ldots -n_{i,k_i})}(F)\subset \prod_{j=0}^{k_i} 
{\rm Gr}_{n-n_{i,0}-\ldots -n_{i,j}}(F)\, .\label{product}\end{equation}
Representing an element $(E,\vec g)$ by $(\alpha, \beta)$, and 
elements of ${\ca O}_{-\sqrt{-1}\delta_i}$ (flags) by elements 
$\gamma_{i,j}, j = 1,\cdots ,k_i$ of $\Lambda^{(n_{i,0}+\ldots+ 
n_{i,j})}(\bbc^n)$, the quotienting is achieved precisely as above:
$$ (\beta_i, \gamma_{i,j})\,\longmapsto\,
\Pi(I(\gamma_{i,j})(\beta_i)) \,\,{\buildrel{\rm def}\over{=}}\,\, 
{\eta_{i,j}}\in \Lambda^{n-n_{i,0}-\ldots -n_{i,j}}(V^*)\, .$$
The stability condition is that these elements are non-zero. For 
elements of ${\rm Gr}_n(V\oplus \bbc^n)$ corresponding to maps 
$f\,:\,\bbc^n\,\longrightarrow\, E_{p_i}$, this map simply takes 
the flag $h_i$ 
defined by $\gamma_{i,j}, j = 1,\cdots ,k_i$, to the flag 
$f(h_i)$. 

The resulting elements $(\alpha, \eta_{i,j})$ are precisely the defining elements of a bundle with quasi-parabolic structure. On the level of line bundles, the map 
$$\psi_i: [{\rm Gr}_n(V\oplus \bbc^n) \times Fl_i(\bbc^n)]_s\,\longrightarrow\,
Fl_i(V)$$
that we have defined (the subscript $s$ denotes the stable locus) pulls back the standard positive line bundles
$L_j \,=\,\Lambda^j(Taut_j^*)$ on the factors ${\rm Gr}_j(V)$ in 
\eqref{product} to the tensor product $L_n\boxtimes L_j $ on
${\rm Gr}_n(V\oplus \bbc^n) \times Fl (\bbc^n)$, where $L_n$ is the standard ample line bundle on ${\rm Gr}_n(V\oplus \bbc^n)$. To see this, one pulls 
back a divisor representing $L_j$: the divisor of planes $g$ 
meeting a fixed $j$ plane $g'$ nontrivially. We take $g'$ to 
correspond to a $j$-plane ${\widehat g}'$ in $E_{p_i}$. Over the set of 
planes in ${\rm Gr}_n(V\oplus \bbc^n)$ corresponding to maps $f\,:\, 
\bbc^n\,\longrightarrow\, E_{p_i}$, this pull-back divisor is 
given for $h\in {\rm Gr}_{n-j}(\bbc^n)$ by the constraint $f(h)\cap 
{\widehat g}'\,\neq\, 0$.

\begin{prop}\label{pp}
Let $(\alpha, \beta)$ correspond to $(E,\vec g)\in \cagm$, with
$\mu_p(\alpha, \beta)\,=\, 0$, and suppose that it lies in a closed orbit. The moment map $\mu'_{n,\ell} $ above applied to $(\alpha, \beta)$ gives an element
$\sqrt{-1}(\delta_1,\cdots ,\delta_\ell)$ of $u(n)^\ell$ with eigenvalues
$\sqrt{-1}\delta_{i,j}$. Then the parabolic bundle associated to $(\alpha, \beta)$ is parabolic 
polystable, for the weights $\delta_{i,j}$.
\end{prop}
 
\begin{proof}
Let us first suppose that $(\alpha, \beta)$ is a stable element.
Let us suppose that $\mu'_{n,\ell}(\alpha,\beta) \,=\, 
\sqrt{-1}( \delta_1,\cdots, \delta_\ell)$, with 
$\sqrt{-1}\delta_j$ belonging to a 
coadjoint 
orbit ${\ca O}_{\sqrt{-1}\delta_i}$. Then
$$(\alpha,\beta,-\sqrt{-1}\delta_i)\in (\mu_p\times \mu'_{n,\ell})^{-1} (0) $$
and so is semistable for the action of ${\rm S(GL}(p)\times \GL(n)^\ell)$; moreover, its orbit is closed.
We 
can then quotient by the action of $S(\GL(n)^{\ell})$; the result is still 
${\rm SL}(p)$-semistable.

Before declaring that we are done, we must check that the 
polarizations match on both sides of $\times_{i=1}^{\ell} \psi_i$.
Recall that $Q\,\supset \,{\rm Gr}_n(V\oplus \bbc^n) $; 
the restriction of ${\mathcal O}_Q(1)$ to ${\rm Gr}_n(V\oplus \bbc^n) $ is its standard positive line bundle (the dual of the top exterior power 
of the tautological bundle), we denote it again by ${\ca 
O}(1)_Q$. The 
map $\psi_i$, as we saw, pulls back $ \bigotimes_{k} {L}_{i,k}^{\alpha_{i,k}}$ to the line bundle $({\ca O}(1)_Q^{\sum_k\alpha_{i,k}})\otimes 
(\otimes_{i,k} L_{i,k}^{\otimes \alpha_{i,k}})$. 
Let $\delta_i^j, j=0,\cdots , k_i+1$, be the eigenvalue corresponding to 
the block of size $n_{i,j}$ in $\delta_i$. Note that $\delta_i^0 = 1/2, \delta_i^{k_i+1} = -1/2$. The standard choice of polarization on $P\times (\prod_i Fl_i(V))$ for parabolic bundles is the bundle 
$${\ca O}_P(1)^\rho\boxtimes (\boxtimes_i (\otimes_{j=0}^{k_i}
L_{n-n_{i,0}-n_{i,1}-\ldots 
-n_{i,j}}^{(\delta_i^j-\delta_{i}^{j+1})})),$$
for $\rho=k-g+1/2$. (In \cite{Bh1, MS}, one has $\rho=k-g$; however, one can check that shifting the parabolic weights from $[0,1]$ to $[-1/2, 1/2]$ requires the change to $\rho=k-g+1/2$.) Now pull this back to 
$Z\times \prod_i{\ca O}_{\sqrt{-1}\delta_i}\subset P\times \prod_i (Q\times {\ca O}_{\sqrt{-1}\delta_i})$. The result is 

 $${\ca O}_P(1)^\rho\boxtimes (\boxtimes_i {\ca O}_Q(1))\boxtimes 
(\boxtimes_i(\otimes_{j=0}^{k_i}
L_{n-n_{i,0}-n_{i,1}-\ldots 
-n_{i,j}}^{(\delta_i^j-\delta_{i}^{j+1})}))\, .$$
The last term is the standard polarization on the coadjoint orbit $\prod_i{\ca O}_{-\sqrt{-1}\delta_i} $; one has the correct line bundles for the $Q$ factors, and the line bundles on $P$ match also.

As our element, after projection, lies in $\mu_p^{-1}(0)$, it satisfies the standard parabolic semistability criterion;  as it has a closed orbit, it will in fact be parabolic polystable.

Now suppose that the element $(\alpha, \beta)$  is semi-stable, but not stable. In 
the latter case one finds that it is a semi stable sum, and that the parabolic bundle it defines is polystable. This is a consequence 
of the orbit being closed. Indeed, let $(\alpha, \eta_{i,j})$ represent 
a semistable parabolic bundle $E$, with $W\subset V$ representing a 
destabilizing subbundle $E'$. One then has a subspace $W^\perp \subset 
V^*$; choose a complementary subspace $U$. At each point of $X$, the 
element $\alpha$ is represented by a product $w_1\wedge\cdots\wedge 
w_k\wedge (w_{k+1} + u_{k+1})\wedge\cdots\wedge(w_n+u_n)$, $w_i\in 
W^\perp, u_i\in U$. One can act by $\bbc^*\subset Sl(V)$ so that projectively, the 
limit element is $w_1\wedge\cdots\wedge w_k\wedge
u_{k+1}\wedge\cdots\wedge u_n$. For the bundle, this amounts to 
rescaling the extension class of $E'\longrightarrow E\longrightarrow 
E/E'$ to zero. Similar considerations hold for the $\eta_{i,j}$, so that 
the limit object is a sum of parabolic bundles. In other words, if an 
extension class (in the sense of parabolic bundles) is nontrivial, 
then the orbit is not closed. Once again, one has that our element, after projection, lies in $\mu_p^{-1}(0)$,  with a closed orbit,  is polystable, and so in particular satisfies the standard parabolic stability criterion.\end{proof}

\begin{prop} The map $C$ is bijective. \end{prop}

\begin{proof} It suffices to consider this over each $(\mu'_{n,\ell})^{-1}(\delta_i)$. For surjectivity, note first that we have seen in \ref{pp} that an element $(\alpha, \beta)$ in $\mu_p^{-1}(0)$ defines a semi-stable parabolic structure; its weights are given by the eigenvalues of its image under $(\mu'_{n,\ell})^{-1}$. Let us first concentrate on the 
case when all the $s_i$ are zero, so that there are no weights equal to 
$1/2$. This implies that the spread of the weights is less than one 
and then the Mehta-Seshadri theorem for parabolic bundles (see \cite{MS}, \cite{Bi}, \cite{Po}) 
give us the flat connection, with the right residues at the puncture. 

Now suppose that there is a subspace $F_0$ of $E_{p_i}$ with weight $1/2$. 
We take a Hecke transform $\widetilde E$ as the subsheaf of 
sections of $E(p)$ whose polar part lies in $F_0$. 
This Hecke transforms does not affect stability. 
Indeed, let $1/2 = \delta'_{i,0}, \delta'_{i,1}, \cdots ,\delta'_{i,r}$ be the distinct elements among the eigenvalues 
$\delta_{i,1}, \cdots, \delta_{i,1}$ of $\delta_i$ with $1/2> \delta'_{i,1} > \cdots > \delta'_{i,r}$. Let
$$0 \subset F_0\subset F_1\subset \cdots\subset F_r= E_p$$ the corresponding flag with weights 
 $$1/2 > \delta'_1 > \cdots > \delta'_r\, .$$ 
 Instead of this flag one has the flag 
 $$ 0 \subset F_1/F_0\subset \cdots\subset F_r/F_0 \subset \widetilde{E}_p\, .$$ 
with weights 
 $$ \delta'_1 > \cdots > \delta'_r \geq -1/2\, .$$ 
One has that the parabolic degree of $E$ equals the parabolic 
degree of $\widetilde{E}$ 
(with the shifted weights) and the same holds for subbundles of $E$. Hence $E$ is parabolic semistable if and only if $\widetilde E$ is, and the same holds for polystability. Now again the spread of the weights is less 
than one, and one can use the result on parabolic bundles to produce a 
flat connection for $\widetilde E$; shifting back (in the space of flat 
connections, taking a Schlesinger transformation) gives us the 
flat connection we want on $E$, with the right residues.

This means that on the level of quotients by $U(n)^\ell$, the map $C$ is surjective. The normal form lemma \ref{normalform} tells us that once we have fixed the bundle, and the (orbits of) the values $\delta_i$ of the moment map, the intersection of the inverse image $\mu_{n,\ell}^{-1}(\delta_i)$ with the elements in $\cagm$ corresponding to the fixed bundle lies a single $U(n)^\ell$-orbit; the same holds for $GM_{n,\delta_0}$.   
The bijection then follows from the fact that the map $C$ is $U(n)^\ell$ equivariant, and that the normal form lemma \ref{normalform} tells us that for an element $a$ in $GM_{n,\delta_0}$, both  $a$ and $C(a)$ have the same stabilizers in $U(n)^\ell$.\end{proof}

As we have a continuous, indeed smooth,  bijection from a compact space, we have, in sum:

\begin{thm}
There is a smooth homeomorphism: 
$$C:GM_{n,\delta_0} \rightarrow \cagm$$
that commutes with the natural actions of $U(n)^\ell$ and indeed of $Gl(n,\bbc)^\ell$ on both terms. The $U(n)$-actions are Hamiltonian, and commute with the moment maps:
\begin{equation} \begin{matrix} GM_{n,\delta_0}&{\buildrel{C}\over{\rightarrow}}&\cagm\\ \\ \downarrow \mu_{n,\ell} &&\downarrow \mu'_{n,\ell} \\ \\ \gu(n)^\ell&=& \gu(n)^\ell\end{matrix}\end{equation}
The symplectic reductions on both sides at $\delta_1,\cdots ,\delta_n$ are the  moduli spaces $PM_{\underline \delta}$ of parabolic bundles with weights given by the eigenvalues $\delta_{i,j}$ of $\delta_{i}$.\end{thm}
 
\section{Generalized parabolic bundles}\label{sec5}

\subsection{Bundles on nodal curves and generalized parabolic structures} Given a bundle $E$ on $X$ with framings at $\ell$ pairs of points $(p_i, q_i)$, there is a natural way of associating to it a bundle on the singular curve $\hat X$ given by identifying the points $p_i,q_i$ of each pair: the framings allow us to identify $E_{p_i}$ with $E_{q_i}$. Alternately, this identification gives, via its graph, a plane in $E_{p_i}\oplus E_{q_i}$, and more generally, one could hope for a way of associating to an element of $\cagm$ a pair consisting of a bundle $E$ and a vector $\vec g$ of $n$-planes $g_i$ in $E_{p_i}\oplus E_{q_i}$.
 
 Pointwise, this procedure is fairly clear: let us consider a bundle $E$ 
with vectors $\vec g^p = (g^p_1,\cdots ,g^p_\ell), \vec g^q = 
(g^q_1,\cdots ,g^q_\ell)$ with $g^p_i\in {\rm Gr}_n(E_{p_i}\oplus 
\bbc^n), g^q_i\,\in\, 
{\rm Gr}_n(E_{q_i}\oplus \bbc^n)$. Let $$R^p_i\,:\, E_{p_i}\oplus 
\bbc^n \,\longrightarrow\,\bbc^n ~\text{~and~}~ R^q_i\,:\,E_{q_i}
\oplus \bbc^n \,\longrightarrow \,\bbc^n$$ be the projections 
and $R_i= R^p_i \oplus R^q_i$. We suppose that:
 \begin{align}
 R^p_i(g^p_i) + R^q_i(g^q_i) \,=\,& \bbc^n\label{condition1}\\
 g^p_i\cap\bbc^n\cap g^q_i \,=\,& 0\, .\label{condition2}
 \end{align}
Let 
$$ g_i = \{(e_p,e_q)\in E_{p_i}\oplus E_{q_i}\, \mid\, \exists ~ 
b\in\bbc^n \ {\rm with} \ (e_p,b)\in g^p_i, (e_q,b)\in g^q_i \}.$$ 
In other words,
$$g_i= R_i((g^p_i \oplus g^q_i)\cap (E_{p_i}\oplus E_{q_i} \oplus 
\Delta))\, , $$
where $\Delta$ denotes the diagonal in $\bbc^n \oplus \bbc^n$. Note that
\eqref{condition1} tells us that $g^p_i \oplus g^q_i$ and $E_{p_i}\oplus 
E_{q_i} \oplus \Delta$ span the full space $E_{p_i}\oplus E_{q_i} \oplus \bbc^n\oplus \bbc^n$;
the intersection of these two spaces is then $n$-dimensional. The 
projection $g_j$ of this intersection to $E_{p_i}\oplus E_{q_i}$ is also $n$-dimensional, as it intersects the kernel of the projection map trivially, by \eqref{condition2}.
We note that $g_i$ is invariant under the diagonal action of 
${\rm GL}(n)$ on the $\bbc^n$s associated to $p_i, q_i$. 

For example, if $g^p_i \,\in\, {\rm Gr}_n(E_{p_i}\oplus \bbc^n)$ is the 
graph of $f^p_i \,\in \,Hom(E_{p_i}, \bbc^n)$ and $g^q_i \,\in\, 
Gr_n(E_{q_i}\oplus \bbc^n)$ is the graph of $f^q_i \in Hom(\bbc^n, 
E_{q_i})$, then $g_i\in {\rm Gr}_n(E_{p_i}\oplus E_{q_i})$ is the graph 
of $f^q_i\circ f^p_i \,\in\, {\rm Hom}(E_{p_i},E_{q_i})$.
 
Algebraically, in terms of the data which defines $(E,\vec g^p, \vec 
g^q)$ as in Section 2, one has the natural map $\Lambda^n 
(V^*\oplus 
\bbc^n) \times \Lambda^n (V^*\oplus \bbc^n) \longrightarrow \Lambda^{2n} 
(V^*\oplus \bbc^n)$. Since 
$(\Lambda^n V^*\otimes \Lambda^n\bbc^n)$ is a direct summand of 
$\Lambda^{2n} (V^*\oplus \bbc^n)$, we have a projection map 
$\Lambda^{2n} (V^*\oplus \bbc^n)) \longrightarrow \Lambda^n V^*\otimes 
\Lambda^n \bbc^n$. 
The restriction of the composite of these two maps gives the map described above. 
If $(\alpha, \beta^p_1, \cdots, \beta^p_{\ell}, \beta^q_1,\cdots,\beta^q_{\ell})$ is the algebraic data encoding $(E,\vec g^p,\vec g^q)$, then $(\alpha,\beta), 
\beta=(\beta_1, \cdots, \beta_{\ell})$ encodes $(E,\vec g)$, where 
$$\beta_j \,= \,i(v)(\beta^p_j\wedge\beta^q_j) \in \Lambda^n(V^*)\, .$$
Here $v$ is the (co)volume element in $\Lambda^n(\bbc^n)$. The constraint $R^p_i(g^p_i) + R^q_i(g^q_i) = \bbc^n$ ensures that $\beta^p_i\wedge\beta^q_i$ is non-zero, and the constraint $g^p_i\cap\bbc^n\cap g^q_i = 0$ then tells us that $\beta_i$ does not vanish.

On the level of moduli spaces, there is a space, constructed in 
\cite{Bh2}, classifying pairs $(E, \vec g)$. We briefly recall the 
construction. (In fact, the construction in \cite{Bh2} is more general, 
but we restrict our attention to this particular case.)
As above, fix disjoint divisors $D_i=p_i+q_i, i=1,\cdots,\ell$, where 
$(p_i, q_i)$ is a pair of distinct points of $X$. Let $E$ denote a vector bundle of rank $n$, degree $\delta_0$ on $X$.
Let $g_i \subset E_{p_i}\oplus E_{q_i}$ be an $n$-dimensional subspace 
and $\vec g= (g_1,\cdots , g_{\ell})$. Such a pair $(E, \vec g)$ is called a 
generalized parabolic bundle (with parabolic structure over the divisors 
$D_i$); let us abbreviate to GPB. 

There is a notion of (semi)stability of GPBs, analogous to that of 
parabolic bundles: as for parabolic bundles, there are weights, and the 
relevant ones here are $\alpha_1 = 1/2, \alpha_2 = -1/2$. For a subbundle $E'$ of $E$ of rank $n'$, degree $\delta_0'$, we set the parabolic degree to be 
\begin{align}{\rm pardeg}(E') =& \delta_0' + \sum_{i=1}^{\ell} 
\alpha_1\dim((E'_{p_i}\oplus E'_{q_i})\cap g_i) + \alpha_2 \dim(E'_{p_i}\oplus E'_{q_i}/((E'_{p_i}\oplus E'_{q_i})\cap g_i))\nonumber\\
 =& \delta_0' + \sum_{i=1}^{\ell} (\dim((E'_{p_i}\oplus E'_{q_i})\cap 
g_i) -n')\, .
 \end{align}
 The definition of semistability is then the usual one, using the parabolic degree to define slopes.

Let $R, \ca{E}, \ca{E}_{p_i}, \ca{E}_{q_i}, \tilde{R}$ 
be as in Section 2.2 (but with half the points $p$ relabelled as $q$).
For $k$ sufficiently large, $R$ contains the underlying bundles of all semistable GPBs. 
Let
$${\rm Gr}_n(\ca{E}_{p_i}\oplus \ca{E}_{q_i}) \longrightarrow R $$ 
be the Grassmannian bundle 
whose fibers are isomorphic to the Grassmannian of $n$-planes in the sum $\ca{E}_{p_i}\oplus \ca{E}_{q_i}$.
Let ${\rm Gr}_n(\ca{E})$ be the fiber product of
${\rm Gr}_n(\ca{E}_{p_i}\oplus \ca{E}_{q_i}), i=1,\cdots, \ell$, 
over $R$. We denote the total space of ${\rm Gr}_n(\ca{E})$ 
by $R^{gpar}$. A point of $R^{gpar}$ corresponds to a 
GPB $(E,\vec g)$. The moduli space ${\ca M}^{gpar}(n,\delta_0)$ is 
a GIT-quotient of $R^{gpar}$ by ${\rm SL}(p)$.

\subsection{Symplectic version} As in Section 3, one can 
build a symplectic version of the moduli space of GPBs. The starting 
point is again the space ${EM}_{n}$ of flat connections on the 
complement of the points $p_i, q_i$, framed at the punctures. We had a 
symplectic action of $U(n)^{2\ell}$ on ${EM}_{n}$, via the framings. 
One has $U(n)$ acting simultaneously on the framing at $p_i$ and on the 
Grassmannian ${\rm Gr}_n(\bbc^n\oplus \bbc^n)$, acting here on the first 
$\bbc^n$; similarly, one has an action of $U(n)$ on the framing at 
$q_i$, and on the Grassmannian ${\rm Gr}_n(\bbc^n\oplus \bbc^n)$, now 
acting on the second copy of $\bbc^n$. We take the symplectic quotient
$$M^{gpar}(n) = ({EM}_{n}\times Gr_n(\bbc^n\oplus 
\bbc^n)^\ell)/\!\!/(U(n)^{2\ell}\, .$$
This gives $\delta_{p_i}$ conjugate to $- \delta_{q_i}$, with eigenvalues in the 
interval $[-1/2, 1/2]$; when all the eigenvalues are in $(-1/2, 1/2)$, 
the elements of the Grassmannian are graphs of maps 
$\gamma_i:\bbc^n\longrightarrow \bbc^n$, which conjugate $\delta_{p_i}$ 
to   $\delta_{q_i}$.
 
As in Section 4, we can define a map
$$C:M^{gpar}(n)\longrightarrow {\ca M}^{gpar}(n,0)\, ;$$
retracing the steps of Section 4, one should be able to show that this 
is an isomorphism, but we will leave this discussion for elsewhere. 

This map indicates what the Narasimhan-Seshadri correspondence should be 
for nodal curves. If $p_i, q_i$ in a desingularization $X$ of the curve are the pairs of points corresponding to 
the nodes, semistable vector 
bundles should correspond to singular unitary connections $\nabla$ on 
the punctured curve $ X^*$, with holonomies 
$\exp(2\pi\sqrt{-1}\delta_{p_i}), \exp(2\pi\sqrt{-1}\delta_{q_i})$, with 
$\delta_{p_i}= -\delta_{q_i}$ having eigenvalues in $(-1/2,1/2)$, and 
unitary isomorphisms between the eigenspaces of $\delta_{p_i}$ and those 
of $\delta_{q_i}$, with the corresponding eigenvalues summing to zero.

\subsection{Relations between $\cagm$ and ${\ca 
M}^{gpar}(n,\delta_0)$}
 
Let us now consider the moduli space $\cagm$ 
for the $2\ell$ marked points $p_i, q_i$; as above, we write an element 
of this moduli space as a triple $(E, \vec g^p, \vec g^q)$. The group 
${\rm GL}(n)^\ell$ acts on $\cagm$, with the $i$-th copy of ${\rm 
GL}(n)$ acting diagonally on the $\bbc^n$s associated to $p_i, q_i$. We 
consider the quotient $\cagm/\!\!/{\rm GL}(n)^\ell$, with the natural 
polarization on the product of Grassmannians
${\rm Gr}_n(V\oplus \bbc^n)$.

\begin{lem}\label{le5.1}
The condition \eqref{condition2} above, $g^p_i\cap \bbc^n\cap g^q_i = 
0$, is a consequence of semistability for the action of ${\rm 
GL}(n)^\ell$.
\end{lem}

\begin{proof}
If $g^p_i\cap \bbc^n\cap g^q_i $ is non empty, let $e_1$ be a non-zero element of the intersection, and complete to a basis $e_i$ of $\bbc^n$. the elements $\beta_{p_i},\beta_{q_i}$ of $\Lambda^n(V^*\oplus \bbc^n)$ describing the Grassmannian framing are of the form 
$$\beta_{p_i} \,=\, b_1\wedge 
(c_2e^*_{2}+b_{2})\wedge\cdots\wedge(c_ne^*_n+b_n),\, \beta_{q_i} \,=\, 
b'_1\wedge 
(c'_2e^*_{2}+b'_{2})\wedge\cdots\wedge(c'_ne^*_n+b'_n)\, .$$
Here the $c_j, c'_j$ are constants, and $b_j, b'_j$ are elements of 
$V^*$. One can then take a $1$-parameter subgroup of $S({\rm 
GL}(p)\times {\rm GL}(n)^\ell)$ taking $\alpha, \beta_{p_i}, 
\beta_{q_i}$ to zero, essentially by putting 
positive weight on $e_1^*$, and negative weight everywhere else. 
\end{proof}

Lemma \ref{le5.1} implies that on the semi-stable locus, the projection 
$$R_i((g^p_i 
\oplus g^q_i)\cap (E_{p_i}\oplus E_{q_i} \oplus \Delta))$$ will always 
be at least $n$-dimensional, and if one defines the variety
\begin{align} \label{spaceZ} Z = \{ ( (E', \vec g^p, \vec g^q), (E,\vec 
g))\,\in\, 
&(\cagm/\!\!/{\rm GL}(n)^\ell) \times {\ca M}^{gpar}(n,\delta_0)\,\mid 
\nonumber 
\\ &E=E', g_i\subset R_i((g^p_i \oplus g^q_i)\cap (E_{p_i}\oplus E_{q_i} 
\oplus \Delta))\}\, ,\end{align}
one obtains a closed subvariety of the product.

Define 
$$(\cagm)_{gen}\,=\, \{(E,\vec g^p,\vec g^q) \in \cagm \,\mid\, 
R^p_i(g^p_i) + R^q_i(g^q_i) 
= \bbc^n \}\, .$$
One has that the variety $Z$ over   the quotient of 
the semi-stable locus of $(\cagm)_{gen}$ by $S({\rm GL}(n)^{\ell})$, is the graph of a 
morphism $\phi$, provided that the image is semistable as a generalized parabolic bundle. 
Note that indeed, this quotient has the correct dimension $n^2(g-1)+1+n^2 \ell =$ dim ${\ca 
M}^{gpar}(n,\delta_0)$.

Specializing a bit further, let us consider the locus 
$$(\cagm)_{gen,0}\,=\, \{(E,\vec g^p,\vec g^q) \in \cagm \,\mid\, 
s_{p_i}= s_{q_i}= t_{p_i}= t_{q_i}=0\}$$
of framed bundles within $\cagm$.

\begin{prop} Let $(E,\vec g)$ be an element of ${\ca 
M}^{gpar}(n,\delta_0)$ for which the planes $g_i$ are the graphs of 
isomorphisms. If $(E,\vec g)$ is stable, then there is a unique 
element
$(E,\vec g^p,\vec g^q)$ of $(\cagm)_{gen,0}/\!\!/({\rm 
GL}(n)^\ell)\,\subset\, \cagm/\!\!/({\rm GL}(n)^\ell)$ corresponding to 
it in the variety $Z$, so that $\phi(E,\vec g^p,\vec g^q) \,=\, 
(E,\vec 
g)$. 

If $(E,\vec g)$ is only semistable, then the same holds, provided that $n$ is odd, or that
$\delta_0 \geq\, (\ell-1)n/2.$
\end{prop}

\begin{proof}
Let $g_i$ be the $i$-th element of $\vec g$; it 
corresponds to a homomorphism $\rho_i: E_{p_i}\longrightarrow E_{q_i}$. 
There are 
elements $ g_i^p, g_i^q$ corresponding to linear maps $\rho_{p_i}: 
E_{p_i}\longrightarrow \bbc^n$, $\rho_{q_i}: E_{q_i}\longrightarrow 
\bbc^n$ with $\rho_i = \rho_{q_i}^{-1}\circ \rho_{p_i}$; these elements 
are unique up to the action of ${\rm GL}(n)$. 

If the element $(E,\vec g)$ is stable, one has, for a subbundle of rank 
$n'$,
$$0 \ <\ \frac{\delta_0}{n} + \frac {-\delta_0' + \sum_{i=1}^{\ell} 
(n'- \dim((E'_{p_i}\oplus E'_{q_i})\cap g_i) )}{n'}\, .$$
The dimension of the intersection $\dim((E'_{p_i}\oplus E'_{q_i})\cap g_i)$ is bounded below by $\max(0,2n'-n)$, giving 
$$0 \ <\ \frac{\delta_0}{n} + \frac {-\delta_0' + \sum_{i=1}^{\ell} (n'-\max(0,2n'-n) )}{n'}$$
and so, for $2n'\geq n$, 
$$0 \ < \ \frac{\delta_0}{n} - \frac {\delta_0'}{n'} + \frac{\ell (n-n')}{n'}$$
and for $2n'\leq n$,
$$
0 \ <\ \frac{\delta_0}{n} - \frac {\delta_0'}{n'} + \ell \, .
$$
Both imply
$$0 < \frac{\delta_0}{n} - \frac {\delta_0'}{n'} + \frac{2\ell (n-n')}{n}$$
This is the stability condition \ref{3}  for our Grassmann framings, when 
$s_i, t_i= 0$. 
 
For the semistability case, one has the same inequalities, but not 
strict; on the other hand, for equality to hold throughout, one needs $n'= n/2$, and the additional stability condition  \ref{4} is then guaranteed  by $\delta_0 \geq\, (\ell-1)n/2.$
\end{proof}
 
The moduli spaces $\cagm, {\ca M}^{gpar}(n,\delta_0)$ are compact; this then gives:

\begin{thm}
The correspondence $Z$ of \ref{spaceZ} defines a birational map $\phi$ between 
the two spaces $\cagm/\!\!/({\rm GL}(n))^\ell$ and $ {\ca M}^{gpar}(n,\delta_0)$.
\end{thm}

\subsection {The symplectic point of view} Let us now consider the space $GM_{n,\delta_0}$ for punctures $p_1,..,p_\ell, q_1,..,q_\ell$, and fix the degree $\delta_0$   to be zero.
From a symplectic point of view, one has the action of the 
diagonal $U(n)^\ell$ in $(U(n)\times U(n))^\ell$ on 
$GM_{n,\delta_0}$, with the $i$-the copy of $U(n)$ acting simultaneously at $p_i$ and $q_i$. Taking the symplectic quotient, the values of 
the moment maps at the punctures $p_i, q_i$ take 
opposing values; one then quotients by the diagonal action of 
$U(n)$ for each pair.

This quotient, over the generic set $S$ corresponding to $s_i= t_i= 0$, so that the elements  $g_{p_i},g_{q_i}$ of the Grassmannian are graphs of invertible maps $\gamma_{p_i}:E_{p_i}\rightarrow \bbc^n$, $\gamma_{q_i}:E_{p_i}\rightarrow \bbc^n$, gives us in a natural way an element of 
$M^{gpar}(n,0)$, simply by taking  the graphs of $\gamma^{-1}_{q_i}\circ\gamma_{p_i}$. The moment map constraint gives us 
$$ (\bbi-\gamma_{p_i}\gamma_{p_i}^*)(\bbi+\gamma_{p_i} \gamma_{p_i}^*)^{-1}=  -(\bbi-\gamma_{q_i}\gamma_{q_i}^*)(\bbi+\gamma_{q_i} \gamma_{q_i}^*)^{-1},$$
from which one deduces that  $\gamma^{-1}_{q_i}\circ\gamma_{p_i}$  conjugates $\delta_{p_i}$ to $-\delta_{q_i}$. 
 We thus have a map 
$$S /\!\!/U(n)^\ell \,\longrightarrow\, M^{gpar}(n)\, .$$ 

Going back to the construction of both spaces from $EM_{n}$, one has 
$$GM_{n,0}/\!\!/U(n)^\ell \,=\, EM_{n}\times 
({\rm Gr}_n(\bbc^n\oplus\bbc^n)\times 
{\rm Gr}_n(\bbc^n\oplus\bbc^n))^\ell/\!\!/(U(n)\times U(n)\times 
U(n))^\ell\, ,$$
where the first copy of $U(n)$ acts on the framing at $p_i$ and on the 
first $\bbc^n$ in the first copy of ${\rm Gr}_n(\bbc^n\oplus\bbc^n)$, 
the second copy of $U(n)$ acts on the framing at $q_i$ and on the first 
$\bbc^n$ in the second copy of ${\rm Gr}_n(\bbc^n\oplus\bbc^n)$, and the 
third copy of $U(n)$ acts on the second $\bbc^n$ in both copies of 
${\rm Gr}_n(\bbc^n\oplus\bbc^n)$. On the other hand,
$$M^{gpar}(n) \,=\, ({EM}_{n}\times {\rm Gr}_n(\bbc^n\oplus 
\bbc^n)^\ell)/\!\!/(U(n)^{\ell}\times U(n)^{\ell})\, ,$$
where the first copy of $U(n)$ acts on the framing at $p_i$ and on the 
first $\bbc^n$ in ${\rm Gr}_n(\bbc^n\oplus\bbc^n)$, and the second copy 
of $U(n)$ acts on the framing at $q_i$ and on the second $\bbc^n$ in 
${\rm Gr}_n(\bbc^n\oplus\bbc^n)$.

The relation between $S/\!\!/U(n)^\ell$ and $M^{gpar}(n)$ is 
thus mediated by the relation between 
$({\rm Gr}_n(\bbc^n\oplus\bbc^n)\times 
{\rm Gr}_n(\bbc^n\oplus\bbc^n))/\!\!/U(n)$ 
and ${\rm Gr}_n(\bbc^n\oplus\bbc^n)$. These two spaces are isomorphic 
over the 
open set consisting of graphs of isomorphisms $\bbc^n\longrightarrow \bbc^n$, 
but the quotient is not an isomorphism away from this.

\section{Examples}\label{sec6}

\subsection{Line bundles}

We now consider line bundles over a curve of arbitrary genus; 
the degree $\delta_0$ will live in a range 
$\{-[(\ell-1)/2],\cdots , [(\ell -1)/2]\}$. There are no 
conditions on subbundles, and one just has the constraints
$\sum_i s_i\leq \ell/2 -\delta_0, \sum_i t_i\leq \ell/2 +\delta_0$. (Thus, 
for $\ell= 1$, we have $t_1 = 0, s_1=0$.)

For the framed moduli, this gives us for $\ell=1$, the Jacobian. 

The other moduli spaces are fibered over the Jacobian. The fiber, over a 
line bundle $L$, is a quotient of $\prod_i{\mathbb P}(L_{p_i}\oplus 
\bbc)$ by 
$\bbc^*$. For $\ell =2$, the fiber is ${\mathbb P}^1$. For $\ell= 3$, 
there are different rational quotients depending on the degree 
$\delta_0$:
\begin{itemize}
\item{} For $\delta_0 = 1$, one has $s_i = 0$, and $\sum_i t_i \leq 2$. This then
gives ${\mathbb P}^2$ as a quotient. 
\item{} For $\delta_0= -1$,the roles of $t_i$ and 
$s_i$ are inverted, so that $t_i = 0$, and $\sum_i s_i \leq 2$. Again, the quotient is ${\mathbb P}^2$.
\item{} For $\delta_0 = 0$, 
one has that $\sum_is_i\leq 1$, $\sum_it_i\leq 1$, and one gets the quotient
${\mathbb P}^1\times{\mathbb P}^1$. \end{itemize}

{}From the symplectic point of view, one has elements of the moduli 
space $EM_1$ given by
elements $A_j\, ,B_j$, $j=1,\cdots ,g$, $C_i$, $i=1,\cdots ,\ell$, of 
$S^1$ and $\delta_i$, $i =1,\cdots ,\ell$, of $\bbr$. One can amalgamate 
the $C_i$ and the $\delta_i$ into elements $\gamma_i$ of $\bbc^*$; the 
space $GM_1$ is obtained by taking the cylinders $\log(|\gamma_i|)\,\in 
\, [-1/2,1/2]$, and collapsing the boundaries of the cylinders to obtain 
spheres. The fiber over the Jacobian is then a product of spheres, which one quotients symplectically by the circle action. One 
fixes the sum of the $\delta_i$ to minus the degree, then quotients by 
$S^1$; different level spaces give different spaces.

\subsection{The genus zero, two point case.}

This is related to the work of Martens and Thaddeus \cite {MaTh}. Let us consider this first from the symplectic side. In this case, the 
geometric data for the space $EM_n$ simplifies somewhat: one has a flat 
connection $\sqrt{-1}\delta d\theta = -\sqrt{-1}\delta d(-\theta)$ on a 
cylinder, and framings $f_1, f_2$ at each puncture. Assuming that the 
connection is expressed in the basis given by the framing $f_1$, the 
data is simply the Hermitian matrix $\delta$ and a unitary matrix $U$ 
expressing the second framing in terms of the first: $f_2 = U\cdot f_1$. 
Going to the Grassmann framed space, one has elements $g_1,g_2$ of ${\rm
Gr}_n(2n)$ satisfying $\mu_1(g_1) \,=\, \sqrt{-1}\delta$, $\mu_1(g_2) \,=\, 
U(-\sqrt{-1}\delta )U^{-1}$. This is to be considered modulo the 
action of two copies of $U(n)$, one at each puncture. One of these copies simply undoes the action of $U$, and one then has the description of the moduli space as the set of pairs $( g_1, g_2)$, satisfying $\mu_1(g_1) = -\mu_1(g_2) $, modulo $U(n)$; in other words, one has the symplectic quotient 
$$({\rm Gr}_n(2n)\times {\rm Gr}_n(2n))/\!\!/U(n)\, ,$$
 under the diagonal action of $U(n)$.

On the open set of planes that are graphs of linear isomorphisms 
$\gamma_1, \gamma_2$, this gives the 
constraint
$$\delta \,=\, (-\bbi + \gamma_1^*\gamma_1)(\bbi+\gamma_1^* 
\gamma_1)^{-1} \,=\, -(-\bbi + 
\gamma_2^*\gamma_2)(\bbi+\gamma_2^*\gamma_2)^{-1}$$
 with the equivalence
$$(\delta, \gamma_1,\gamma_2) \,\longrightarrow\, (U\delta U^{-1}, 
\gamma_1U^{-1},\gamma_2U^{-1})\, .$$
Decomposing into a product of a positive Hermitian part and a unitary part, one has $\gamma_1 = (\gamma_1\gamma_1^*)^{1/2}\cdot (\gamma_1\gamma_1^*)^{-1/2} \gamma_1$, and one can use the unitary action to normalize 
$\gamma_1$ to a positive Hermitian matrix. One then has that $\gamma_1 = 
(\gamma_1\gamma_1^*)^{1/2}$, and via the relation above, $\gamma_1$ is 
then computable in terms of $\gamma_2$. The open set of the 
moduli space is simply given by the possible choices for the matrix $\gamma_2$, and so is $\GL(n,\bbc)$.

{}From the holomorphic viewpoint, if one restricts to the open set over 
which the bundle is trivial, and for which the planes correspond to the 
graphs of framings, one again finds $\GL(n,\bbc)$: one has a trivial 
bundle $E$, and two invertible linear maps from $H^0({\mathbb P}^1, E) 
= E_p = 
E_q$ to $\bbc^n$. One can use the automorphisms of the bundle to normalize one of the maps to the identity, with the other map giving the element of $\GL(n,\bbc)$.

\subsection{The one point case.}

Let us consider the case of a framing at just one point $p$. The 
degree $\delta_0\in \{-[( n-1)/2]\, ,\cdots\, , [( n-1)/2]\}$. Here 
we have $s_1\leq n/2 - \delta_0$, and $t_1 \leq n/2 + \delta_0$. 
If the base bundle $E$ is stable, the pair $(E,g)$ is also. Over 
the locus of stable (hence simple) bundles, the moduli space 
$\cagm$ has fiber given by the quotient of the Grassmannian by 
$\bbc^*$. This quotient will depend on $\delta_0$; for example, 
if $n$ is odd and $\delta_0 = -( n-1)/2$, then $t_1 = 0, 
s_1<n$; the set of planes is then the set of graphs of non-zero 
maps from $\bbc^n$ to the fiber of the bundle at $p$, and, 
quotienting by $\bbc^*$, one simply gets the projective space 
${\mathbb P}^{n^2-1}$.

\end{document}